\documentclass[a4paper,12pt,oneside,dvipsnames]{amsart} % reqno if desired

\usepackage{amsmath,amsfonts,amsthm,amssymb}

\usepackage{xcolor}
\usepackage{graphicx}
\usepackage{geometry}
\usepackage{cite}

\usepackage{bbm}
\usepackage{tikz-cd}

\usepackage{hyperref}
\hypersetup{colorlinks=true, citecolor=PineGreen, linkcolor=RoyalBlue, linktoc=page,urlcolor=RoyalBlue}

\usepackage{geometry}
\theoremstyle{plain}
\newtheorem{theorem}{Theorem}[section] %[section] % to omit decimals
\newtheorem{lemmy}[theorem]{Lemma}
\newtheorem{prop}[theorem]{Proposition}
\newtheorem{cor}[theorem]{Corollary}

\theoremstyle{remark}
\newtheorem{rem}[theorem]{Remark}

\theoremstyle{definition}
\newtheorem{defn}{Definition}[section]
\newcommand{\A}{\mathbb{A}}

\newcommand{\R}{\mathbb{R}}
\newcommand{\Rx}{\mathbb{R}^{\times}}
\newcommand{\C}{\mathbb{C}}
\newcommand{\Cx}{\mathbb{C}^{\times}}

\newcommand{\Z}{\mathbb{Z}}
\newcommand{\Zp}{\mathbb{Z}_{p}}

\newcommand{\Q}{\mathbb{Q}}
\newcommand{\Qx}{\mathbb{Q}^{\times}}

\newcommand{\Pb}{\mathbb{P}}

\newcommand{\ggT}{\operatorname{gcd}}
\newcommand{\kgV}{\operatorname{lcm}}

\newcommand{\Ac}{\mathcal{A}}

\newcommand{\Sch}{\mathcal{S}}

\newcommand{\Wh}{\mathcal{W}}

\newcommand{\uh}{\mathfrak{H}}

\newcommand{\of}{\mathfrak{o}}
\newcommand{\ofx}{\mathfrak{o}^{\times}}

\newcommand{\pf}{\mathfrak{p}}
\newcommand{\af}{\mathfrak{a}}

\newcommand{\Hf}{\mathfrak{H}}

\newcommand{\Fx}{F^{\times}}

\newcommand{\GL}{\operatorname{GL}}

\newcommand{\SL}{\operatorname{SL}}

\newcommand{\SLt}{\widetilde{\operatorname{SL}}}
\newcommand{\GLt}{\widetilde{\operatorname{GL}}}

\newcommand{\SO}{\operatorname{SO}}

\newcommand{\Gt}{\widetilde{G}}
\newcommand{\Ht}{\widetilde{H}}
\newcommand{\Zt}{\widetilde{Z}}
\newcommand{\zt}{\tilde{z}}
\newcommand{\Bt}{\widetilde{B}}

\newcommand{\Mt}{\widetilde{M}}

\newcommand{\N}{\mathbb{N}}

\renewcommand{\Im}{\operatorname{Im}}

\newcommand{\ind}{\operatorname{ind}}
\newcommand{\sgn}{\operatorname{sgn}}

\newcommand{\Vol}{\operatorname{Vol}}

\newcommand{\bs}{\backslash}

\newcommand{\Mod}[1]{\ (\operatorname{mod}\ #1)}
\newcommand{\f}{\operatorname{f}}

\newcommand{\dxy}{d^{\times}y}

\newcommand{\eve}{\operatorname{eve}}

\newcommand{\smat}[4]{\left(\begin{smallmatrix}
#1 & #2 \\ #3 & #4
\end{smallmatrix}\right)}

\newcommand{\abs}[1]{\lvert{#1}\rvert}

\newcommand{\otherwise}{\text{otherwise}}

\newcommand{\ifs}{\text{if }}

\newcommand{\emb}{\hookrightarrow}

\title{\texorpdfstring{Vorono\"i}{Voronoi} summation for half-integral weight automorphic forms}

\author{Edgar Assing}
\author{Andrew Corbett}

\newcommand{\Date}{7$^{\mathrm{th}}$ September 2020}
\date{\Date}

\address{Mathematisches Institut, Endenicher Allee 60, D-53115 Bonn, Germany}
\email{assing@math.uni-bonn.de}

\address{The Innovation Centre, University Of Exeter, Exeter, EX4 4RN, UK}
\email{A.J.Corbett@exeter.ac.uk}

\begin{document}

\begin{abstract}

A general Vorono\"i summation formula for the (metaplectic) double cover of $\GL_2$ is derived via the representation theoretic framework \`a la  Ichino--Templier.
The identity is also formulated classically and used to establish Vorono\"i summation formulae for half-integral weight modular forms and Maa\ss\ forms.

%In this note we derive a general Vorono\"i summation formula a la Ichino-Templier for the double cover $\GLt_2$ of $\GL_2$. Classically this is used to establish some Vorono\"i summation formulae for half-integral weight modular forms.

\end{abstract}

\maketitle
% \texorpdfstring{
%\setcounter{tocdepth}{2}
%\tableofcontents

\section{Introduction}

The Vorono\"i summation formula for classical modular forms has a long history. This formula and its generalisations are crucial technical tools in the proofs of many deep results of analytic number theory. One such example is Jutila's treatment of exponential sums involving Hecke eigenvalues \cite{jutila}, which implies a sixth moment bound for the $L$-function of a Hecke eigenform of level $1$ on the critical line. Other applications include shifted convolution problems, spectral reciprocity formulae, estimates of the $L^4$-norm of Maa\ss\  forms and many more.
A rolling history and other applications of Vorono\"i summation are collected in \cite{miller-schmidt-book}.

Such classical Vorono\"i formulae are very well understood. Their manifestation is closely related to the functional equation of the underlying $L$-function and associated Atkin--Lehner theory \cite{KMV-rankin}. However, they can also be formulated more abstractly in the language of automorphic representations for $\GL_2$; see \cite{templier_note, AC_voronoi}. Adopting a representation theoretic stance permits greater generality in describing local aspects, such as non-trivial levels, old forms and local $p$-adic test function. For example, for all $n \geq 2$, the benchmark classical formula for $\GL_n$ \cite{miller-schmid-general} has been completed to the greatest generality permitted by the representation theoretic approach \cite{ichino-templier_voronoi, corbett_voronoi}.

Our goal here is to overhaul the theory for half-integral weight automorphic forms.
From a classical angle, recall the construction of half-integral weight holomorphic modular forms and half-integral Maa\ss\  forms. These correspond to smooth functions of moderate growth on the upper half-plane $\uh=\{z\in\C : \Im(z)>0\}$ which, up to a character $\chi$, transform under the action of the congruence lattice $\Gamma_0(4N)\subset \SL_{2}(\R)$ with respect to the $\vartheta$ multiplier. If $f$ is such a form of weight $k+\frac{1}{2}$, level $4N$, and character $\chi$, then $f$ has a Fourier expansion at the cusp $\infty$, given by
$
f(x+iy)=\sum_{n\in\Z} a_{f}(n) \kappa_f(ny)e(nx).
$
The Fourier coefficients $a_{f}(n)$ have strikingly different properties in contrast to their integral weight contemporaries. One being that, even for Hecke-eigenforms, the associated Dirichlet series $\sum_{n>0}a_f(n)n^{-s}$ is not Eulerian.

Using the modularity properties of a holomorphic half-integral weight modular form $f$, Duke--Iwaniec \cite[Theorem 4]{duke-iwaniec-half} derive a preliminary case of such a formula which they then use to determine cancellation in the sum $\sum_{p\leq X} \chi(p) a_{f}(p)$ for a (principal or primitive) Dirichlet character $\chi$. A different Vorono\"i formula has been given  for Hecke--Maa\ss\ forms by Bykowski\u{\i} \cite{bykowskii}. This formula has recently found an application to the mass distribution of Saito--Kurakawa lifts amongst Siegel modular forms; see \cite{blomer-corbett_symplectic}. These are two specific examples of Vorono\"i-type summation formulae for half-integral weight. In general, such formulae are sparsely available and, to the best of our knowledge, there is no unified approach to be found in the literature.

In this work we give a general Vorono\"i formula which applies to any automorphic form arising from the representation theory of $\GLt_2(\A)$; see Theorem \ref{thm:adelic-Voronoi}.
We also give an explicit classical formulation of the result, applying to both the modular and Maa\ss\ cases, alongside various related formulae in \S \ref{sec:main-classical}. We intend the ``oven ready'' classical formulae to carry sufficient generality whilst remaining explicit enough for applications.

To preview the main classical result (Theorem \ref{thm:classical-Voronoi}),
let $k$ be a non-negative integer, $N$ a positive multiple of $4$ and $\chi$ a Dirichlet character modulo $N$.
Let $f$ denote a cuspidal weight $k+\frac{1}{2}$ autormorphic form of level $N$ and character $\chi$ as described in Definition \ref{def:classical-forms}, which includes (normalised) holomorphic modular forms and Maa\ss\ forms. Let $F\colon\R\to\R$ be a smooth (test) function with compact support in $\R_{>0}$. Further take $b\in \N$ and $a\in \Z$ with $\ggT(a,bN)=1$. Then we have the half-integral Vorono\"i summation formula
\begin{multline}
	\sum_{n\in\Z_{\neq 0}}e\bigg(\frac{an}{b}\bigg) a_{f}(n)  F\left(n\right) = \sum_{n\in\Z_{\neq 0}}e\left(-n\frac{\overline{a}}{b\delta(\mathfrak{b})}\right) a_f(n;\mathfrak{b})[\mathcal{H}_f^{\sgn(n),+}F]\left(\frac{n}{\delta(\mathfrak{b})b^2}\right)
	\nonumber
\end{multline}
where $a_f(n;\mathfrak{b})$ is the Fourier coefficient of $f$ at the cusp $\mathfrak{b} = \frac{a}{b}$ (see \eqref{eq:cusp-fourier-expansion}) and the Bessel transform $\mathcal{H}_f^{\pm,+}F$ of $F$, determined by the archimedean type of $f$ alone, is given in \S \ref{sec:bessel-classical}.

%Along the way we give a concise summary of the prerequisite representation theory of $\GLt$ (see \S \ref{sec:not}) and prove a variety results of independent value related to the Whittaker models and Bessel transforms associated to such representations (see \S \ref{sec:back}).

\subsection{From automorphic forms to metaplectic representations}
Automorphic forms of half-integral weight correspond to automorphic representations of the metaplectic cover $\GLt_2$ of $\GL_2$. The characterising feature of the metaplectic group is that the $2$-group of rotations
$M=\{\pm 1\}\subset \GL_{2}$ is projected onto by the cyclic $4$-group $\Mt\subset\GLt_{2}$. A physicist might understand the $2$-group of rotations by revolving a particle by $\pi$ degrees and then again back to $2\pi$ degrees, completing a full turn. Whereas if the particle were to carry a genuine action of $\Mt$ then it must complete two full turns, a $4\pi$ revolution, before returning to the origin. The particle is said to have the property of `half-integral spin'. In the theory of automorphic forms, the genuine action of $\Mt$-representations is realised though the structure of the related cocycles, originally defined via the $\vartheta$-multiplier system.

Such a theory of representations of $\GLt_{2}$ plays a crucial role in identifying the local-global framework of automorphic forms of `half-integral weight'.
The beautiful line of work of S.~Gelbart and I.~Piatetski-Shapiro \cite{gelbart-ps-distinguished,gelbart-ps-on-shimura,gelbart-ps-Shimura-vs-Walsdpurger} establishes all the basics needed to build a foundational theory on which one may construct Vorono\"i-type formulae.
We give a concise review of this theory in \S \ref{sec:not} and \S \ref{sec:back}. It is on these bases that we are able to prove Theorem \ref{thm:adelic-Voronoi}, our main theorem. This result, derived in \S \ref{sec:main}, provides a flexible adelic identity which may be specialised to give various Vorono\"i-type formula. Finally, in \S \ref{sec:classical} we move to the classical setting and derive to more specific, but still quite general, Vorono\"i formulae. To produce our classical formulae we require the necessary adelisation procedure for half-integral weight forms. Whilst standard in the $\GL_2$ case and certainly well-known to experts for $\GLt_2$, the literature seems only to contain incoherent parts of the story. We give a complete, modern account of the process in \S \ref{sec:classical}.

To conclude this introduction, we sketch the basic Vorono\"i mechanism in the familiar world of `integral spin'; that is, for automorphic representations of $\GL_2$ giving action only to the two-group of rotations $M$, resulting in whole integral weight at infinity. This will serve as a guide by which to navigate the metaplectic situation $M\rightsquigarrow \Mt$ and motivate the constructions we recall in \S \ref{sec:not}-\ref{sec:back}.

\subsection{Strong Gelfand pairs as a source of Vorono\"i summation}
\label{sec:intro-gelfand}

To give insight into the big picture, recall the interpretation of Vorono\"i summation in the setting of strong Gelfand formations for $\GL_2$. This will serve as a guide, since we shall rely on similar ideas in our treatment of the $\GLt_2$ case.

A `strong Gelfand pair' consists of a topological group $G$ and a subgroup $H\subset G$ such that
$$
\dim_{\C}\text{Hom}_H(\sigma,\pi\vert_H)\leq 1 \nonumber
$$
for all irreducible representations $\pi$ of $G$ and $\sigma$ of $H$. We call the diagram
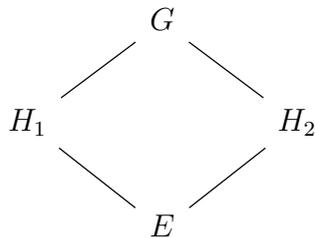
\begin{equation}
\label{eq:strong-gelfand}
\begin{tikzcd}
& G \arrow[dash]{rd} \arrow[dash]{ld} &  \\
H_1 \arrow[dash]{rd} & & H_2 \arrow[dash]{ld}\\
& E & 
\end{tikzcd}\nonumber
\end{equation}
a `strong Gelfand formation' if each contained edge corresponds to a strong Gelfand pair. It is well known that such formations can be used to study periods; that is, elements of $\text{Hom}_E(\pi\vert_E,\delta)$ where $\pi$ and $\delta$ are irreducible representations of $G$ and $E$, respectively. This is achieved by taking spectral expansions along both intermediate sides of the formation.

It was pointed out in \cite{reznikov_rs}, and elaborated on in \cite{templier_note}, that the strong Gelfand formation
\[
\begin{tikzcd}
& G = \GL_2 \arrow[dash]{rd} \arrow[dash]{ld} &  \\
N=\left\{ \left(\begin{matrix} 1 & \star \\ 0 & 1 \end{matrix} \right)  \right\} \arrow[dash]{rd} & & \overline{N}=\left\{  \left(\begin{matrix} 1 & 0 \\ \star & 1 \end{matrix} \right) \right\} \arrow[dash]{ld}\\
& \{ e \} & 
\end{tikzcd}
\]
can be seen as the source of Vorono\"i summation. In the case of the metaplectic cover, $\GLt_2$, the fact that $N^*\subset \GLt_2$ is not a strong Gelfand pair becomes apparent. The formulae presented in this work are instead motivated by the strong Gelfand formation
\[
\begin{tikzcd}
&\GLt_2 \arrow[dash]{rd} \arrow[dash]{ld} &  \\
\Zt N^* \arrow[dash]{rd} & & \Zt\overline{N}^* \arrow[dash]{ld}\\
& \{ 1 \}^* & 
\end{tikzcd}
\]
where $\Zt\subset \GLt_2$ is the pre-image of the centre of $\GL_2$ under $\GLt_2\rightarrow\GL_2$ and $N^*$ a realisation of $N$ inside $\GLt_2$ (see \ref{eq:g-star}). The groups $\Zt N^*\subset \GLt_2$ form a strong Gelfand pair by the uniqueness of an analogue of the Whittaker model in the metaplectic setting; see \ref{sec:back}.

The technique enabling the $\GL_2$ formation to give rise to Vorono\"i summation is sketched below in \S \ref{sec:gltwo-revisited}. We pick up the story for $\GLt_2$ again in \S \ref{sec:main}.

\subsection{The integral weight case for \texorpdfstring{$\GL_2$}{GL(2)} revisited}
\label{sec:gltwo-revisited}

Given a suitable element $\phi\in L^2(G(F)\backslash G(\A_F),\omega)$ which transforms with respect to some irreducible automorphic representation $\pi_{\phi}$ we expand the period $\phi\mapsto \phi(e)$ through the right side of the formation to get
\begin{equation}
	\phi(e) = \sum_{\psi\in \widehat{N(F)}} W_{\psi}(\phi)\cdot \psi(e) \nonumber
\end{equation} 
where $$\phi\mapsto W_{\psi}(\phi) = \int_{N(F)\backslash N(\A_F)} \phi(n)\overline{\psi(n)}dn$$ is the Whittaker period, the adelic realisation of a Fourier coefficient. Of course, so far we have done nothing more than write down the Whittaker expansion in a strange way. Performing the same procedure on the right hand side of the formation results in the identity
\begin{equation}
	\sum_{\psi\in \widehat{N(F)\backslash N(\A)}} W_{\psi}(\phi)\cdot \psi(e) = \sum_{\psi'\in \widehat{\overline{N}(F)}} \tilde{W}_{\psi'}(\phi)\cdot \psi'(e).\label{eq:GL_2_fund_id}
\end{equation}
Here we encounter the complementary periods $$\phi\mapsto \tilde{W}_{\psi'}(\phi) = \int_{\overline{N}(F)\backslash \overline{N}(\A_F)} \phi(\overline{n})\overline{\psi'(\overline{n})}d\overline{n}.$$
Up to this point the formula is essentially a triviality and there are two key steps remaining before deriving a useful Vorono\"i summation formula.

\textit{Step~$1$} is to set up the left-hand side. We introduce the notation $a(y)=\smat{y}{}{}{1}$, as used for matrices throughout. First of all let us observe that we can identify $F^{\times}$ with $\widehat{N(F)\backslash N(\A)}$ by fixing some $N(F)$-invariant character $\psi$ explicitly realised via $$\xi\mapsto \left[n\mapsto\psi_{\xi}(n) = \psi(a(\xi)na(\xi)^{-1})\right].$$ Now we can rewrite
\begin{equation}
	W_{\psi_{\xi}}(\phi)\cdot \psi(e) = W_{\phi}(a(\xi)) = \int_{N(F)\backslash N(\A_F)} [\pi(a( \xi))\phi](n)\overline{\psi(n)}dn. \nonumber
\end{equation}
By multiplicity one for the Whittaker model, the global Whittaker function factors into local ones. We get
\begin{equation}
	W_{\phi}(a(\xi)) = \prod_v 	W_{\phi,v}(a(\xi_v)).\nonumber
\end{equation}
For a finite set of places $S$, including each place at which $\pi$ ramifies and all archimedean places, identify the Schwartz functions $f_v\in \mathcal{S}(F_v^{\times})$ for $v\in S$ given by
\begin{equation}
	W_{\phi,v}(a(\xi_v)) =\abs{\xi}_v^{\frac{1}{2}} f_v(\xi_v).
\end{equation}
This choice could too be made in reverse due to the inclusion of the Schwartz space in the Kirillov model.
For $v\not\in S$ we assume that $W_{\phi,v}(\cdot)$ is the unique spherical element in the local Whittaker model normalised by $W_{\phi,v}(1)=1$. Thus $\phi$ is now completely determined by these data. The right hand side of \eqref{eq:GL_2_fund_id} reads
\begin{equation}
	\sum_{\xi\in F^{\times}} \lambda_{\pi} \left( \frac{(\xi)}{((\xi),\prod_{v\in S_{<\infty}}\mathfrak{p}_v^{\infty})}\right) \prod_{v\in S}f_v(\xi).\nonumber
\end{equation}

Step~$2$ is to explicate the right hand side. Using $\overline{N} = \left(\begin{matrix} 0 & 1 \\ 1 & 0\end{matrix}\right)N\left(\begin{matrix} 0 & 1 \\ 1 & 0\end{matrix}\right)$ and making the correct identifications we can rewrite the period on the right-hand side as
\begin{equation}
	\tilde{W}_{\psi'_{\xi}}(\phi)\cdot \psi'_{\xi}(e) = W_{\phi}\left(\left(\begin{matrix} 0 & 1 \\ 1 & 0\end{matrix}\right)a(\xi)\right) = \prod_v W_{\phi,v}\left(\left(\begin{matrix} 0 & 1 \\ 1 & 0\end{matrix}\right)a(\xi)\right).\nonumber
\end{equation}
For $v\not\in S$ this enables us to evaluate the local Whittaker function simply by using right-$\GL_2(\of)$ invariance. If $v\in S$, one can either use the known transformation behaviour of $W_{\phi,v}$, for example local Atkin-Lehner/new-vector theory, or derive an integral representation for $W_{\phi,v}\left(\left(\smat{ 0} { 1} {1} {0}\right)a(\xi)\right)$ featuring a Bessel kernel and the original function $f_v$. Note that the Bessel-transform can be constructed solely by using Fourier analysis and the uniqueness of $(N,\psi)$ and $(\overline{N},\psi')$ Whittaker models. Nevertheless, it has a well known connection to the local functional equation.

Indeed, it seems that \textit{all} Vorono\"i type formulae concerning Maa\ss\  forms and classical holomorphic modular forms can be deduced from this general strategy using adelisation and de-adelisation.

The reason we discuss this strategy in quite some detail is because we now follow the same ideas in the metaplectic/half-integral weight world. The main issue that arises is that multiplicity one fails for the Whittaker model of representations of $\GLt_2(\A)$ and we constantly have to account for this.

\section{Notation and Metaplectic Constructions}\label{sec:not}

For two integers $a$ and $b$, we write $a|b$ to denote that $a$ divides $b$ and moreover $a|b^\infty$ shall denote that $a|b^n$ for some positive integer $n$. We start counting $\N$ at $\inf\N =1$. We shall use the standard exponential notation $e(z)=e^{2\pi i z}$ for $z\in\C$. For each $z=re^{i\theta}\in\C$ we define a branch cut of the square root by letting $z^{1/2}$ denote the complex number $\sqrt{\abs{z}}e^{i\theta'}$ where $\theta'\in \frac{\arg(z)}{2}+2\pi i \Z$ such that $-\pi/2<\theta' \leq \pi/2$. If $k\in\Z$ then let $z^{k+\frac{1}{2}}:=z^{k}z^{\frac{1}{2}}$.
If $f\colon X\rightarrow \C$ is a function on a monoid $X$ then denote the right-translates of $f$ by $R(y)f(x)=f(xy)$ for each $y\in X$.

%Throughout, we let $G$ denote the algebraic group $\GL_{2}$.
Let $F$ be a ring with unit $1\in F$. We introduce the following matrices in $\GL_{2}(F)$. For $\lambda\in\Fx, x\in F,$ and $y\in \Fx$ let
\begin{equation}
z(\lambda) = \begin{pmatrix}
\lambda& \\
& \lambda
\end{pmatrix};\quad
n(x) = \begin{pmatrix}
1&x \\
&1
\end{pmatrix};\quad
\overline{n}(x) = \begin{pmatrix}
1& \\
x&1
\end{pmatrix};\quad
w=\begin{pmatrix}
&1\\
-1&
\end{pmatrix};
\nonumber
\end{equation}
\begin{equation}
a(y) = \begin{pmatrix}
y& \\
&1
\end{pmatrix}; \quad
\overline{a}(y) = \begin{pmatrix}
1& \\
&y
\end{pmatrix};\quad
a'(y)=\begin{pmatrix}
y&\\
&y^{-1}
\end{pmatrix}
.
\nonumber
\end{equation}
These matrices determine the subgroups $Z(F)=\lbrace z(\lambda) :  \lambda\in \Fx\rbrace$, $N(F)=\lbrace n(x) :  x\in F\rbrace$, $A(F)=\lbrace a(y) :  y\in \Fx\rbrace$, $A_0(F)=\lbrace a'(y) :  y\in \Fx\rbrace$, and $B=ZNA$, omitting the parenthetic $F$ when convenient. We denote the square-roots of unity in $F$ by $Z_{2}=\lbrace \epsilon\in F : \epsilon^{2}=1 \rbrace$. We will also encounter the cyclic group of order four, which we denote by $\widetilde{M}$.

Over the real numbers $\R$ we define the rotation matrix 
\begin{equation}
k(\theta)=\begin{pmatrix}
\cos(\theta)&\sin(\theta)\\
-\sin(\theta)&\cos(\theta)
\end{pmatrix} \nonumber
\end{equation}
and the corresponding group $\SO(2) = \{ k(\theta)\colon \theta\in [0,2\pi)\}$. We further define $A_0^+ =\{ a'(y)\colon y\in \R_{>0} \}\subset A_0(\R)$.

As usual, let $\Gamma_0(N)$ denote the group of matrices $\left(\begin{smallmatrix}a&b\\c&d
\end{smallmatrix}\right) \in \SL_2(\Z)$ such that $c\equiv 0 \Mod{N}$. The subgroup $\Gamma_1(N)$ consists of those matrices in $\Gamma_0(N)$ with the additional property that $a \equiv d \equiv 1 \Mod{N}$.

\subsection{The metaplectic group over local fields}

We specify the local field $F$ of characteristic $0$ that we consider.
If $F$ is archimedean, it is either $\R$ or $\C$.
If non-archimedean we assume $F$ to be a finite extension of $\Q_p$ in which $p$ is the residual characteristic of $F$.

\subsubsection{Definitions}

If $F\neq\C$ then the group $\SL_{2}(F)$ admits a unique, non-trivial two-fold cover $\SLt_{2}(F)$, see \cite{moore,weil-unitary}. This group is known as the `metaplectic group' and is defined by the exact sequence
\begin{equation*}
1\longrightarrow Z_{2} \longrightarrow \SLt_{2}(F)  \longrightarrow \SL_{2}(F)\longrightarrow 1.
\end{equation*}
If $F=\C$ we set $\SLt_{2}(F)=\SL_{2}(\C)\times Z_{2}$, whence the above sequence splits.

Any automorphism of $\SL_{2}(F)$ lifts uniquely to an automorphism of $\SLt_{2}(F)$. In particular, $A=\lbrace a(y) :  y\in \Fx\rbrace$ acts on $\SL_{2}(F)$ by conjugation, lifting uniquely to an action of $A$ on $\SLt_{2}(F)$. We thus define the metaplectic cover of $\GL_{2}(F)$ to be the semi-direct product $\GLt_{2}(F)=\SLt_{2}(F)\rtimes A$, once again obtaining an exact sequence of locally compact groups
\begin{equation}\label{eq:exact-sequence}
1\longrightarrow Z_{2} \longrightarrow \GLt_{2}(F)  \longrightarrow \GL_{2}(F)\longrightarrow 1.
\end{equation}
As before, \eqref{eq:exact-sequence} splits if and only if $F=\C$. If $H\leq \GL_{2}(F)$ then let $\Ht\leq \GLt_{2}(F)$ denote the pre-image of $H$ under $\GLt_{2}(F)\rightarrow \GL_{2}(F)$. We say \eqref{eq:exact-sequence} splits over $H$ if $\Ht\cong H\times Z_{2}$; for example, \eqref{eq:exact-sequence} splits over the subgroups $N$ and $A$ defined as in \S \ref{sec:not}.

\subsubsection{Explicit form of the metaplectic group}

More concretely, arbitrary elements of $\GLt_{2}(F)$ may be denoted as pairs $(g,\epsilon)$ for $g\in \GL_{2}(F)$ and $\epsilon\in Z_{2}$. Multiplication in $\GLt_{2}(F)$ is then given by
\begin{equation}\label{eq:local-multiplication}
(g_{1},\epsilon_{1})\cdot (g_{2},\epsilon_{2}) = (g_{1}g_{2},\beta(g_{1},g_{2})\epsilon_{1}\epsilon_{2})
\end{equation}
where $\beta\colon \GL_{2}(F)\times \GL_{2}(F) \rightarrow Z_{2}$ is the co-cycle defined by
\begin{equation}\label{eq:local-beta}
\beta(g_{1},g_{2})=\left[ \frac{\ell(g_{1}g_{2})}{\ell(g_{1})}, \frac{\ell(g_{1}g_{2})}{\ell(g_{2})} \det(g_{1})\right]s(g_1)s(g_2)s(g_1g_2)
\end{equation}
where $[\cdot , \cdot ]=[\cdot , \cdot ]_{F}$ is the quadratic Hilbert symbol,
\begin{equation}
[a,b] = \left\lbrace \begin{array}{rl}
1 & \ifs z^{2} = a x^{2} + by^{2} \text{ for some non-trivial } (x,y,z)\in F^{3}\\
-1 & \otherwise
\end{array}\right.
\end{equation}
for $a,b\in\Fx$, and $\ell,s\colon \GL_{2}(F)\rightarrow\Fx$ are the functions
\begin{equation}
\ell\left(\begin{pmatrix}
a&b\\
c&d
\end{pmatrix}\right)= \left\lbrace \begin{array}{ll}
c & \ifs c\neq 0\\
d & \ifs c=0
\end{array}\right.\nonumber 
\end{equation}
and
\begin{equation}
	s\left(\left(\begin{matrix} a& b\\ c& d\end{matrix}\right)\right) = \begin{cases} [c,d\det(k)] &\text{ if }cd\neq 0,\, F\neq \R,\C  \text{ and $v(c)$ is odd,}\\
	1 &\text{ else.} \end{cases} \nonumber 
\end{equation}
For each $g\in \GL_{2}(F)$, we introduce the notation
\begin{equation}
\label{eq:g-star}
g^{*}:=(g,1)\in\GLt_{2}(F).
\end{equation}
We shall apply this notation to subgroups $H\leq \GL_{2}(F)$ so that $H^{*}=\{h^{*} : h\in H\}$. Note that $H^*$ is not necessarily a subgroup of $\GLt_2(F)$. We also abuse terminology and let $\epsilon$ denote the element $(1,\epsilon)\in \GLt_2$ for $\epsilon\in Z_{2}$.

\subsubsection{Genuine functions and representations}

We say a function $\Phi$ on $\GLt_{2}(F)$, valued in any module containing $Z_{2}$, is \textit{genuine} if $\Phi(g(1,\epsilon))=\epsilon \Phi(g)$ for all $g\in\GLt_{2}(F)$ and $\epsilon\in Z_{2}$. Otherwise, we say that $\Phi$ \textit{factors through $\GL_{2}(F)$}. In what follows we shall only consider genuine representations of $\GLt_{2}(F)$. 

\subsubsection{The centre}\label{sec:the-centre}

The sequence \eqref{eq:exact-sequence} does not split over the centre $Z$ of $\GL_{2}(F)$. Whilst $\Zt$ itself is abelian, it commutes with the rest of $g\in \GL_2$ via the identity $g^{*}z(\lambda)^{*}=z(\lambda)^{*}g^{*}[\lambda,\det g]$ for all $\lambda\in\Fx$. It is then evident that \eqref{eq:exact-sequence} does split over $Z^{2}=\lbrace z(\lambda^{2})\colon \lambda\in\Fx\rbrace\leq Z
$ and in fact the centre of $\GLt_2$ is given by the pre-image of $Z^{2}$:
\begin{equation*}
\Zt^{2}=Z^{2}\times Z_{2}.
\end{equation*}
For the quotient, we identify the representatives $\Zt/\Zt^{2}=\{z(\lambda)^{*} : \lambda\in \Fx/(\Fx)^{2}\}$. Note that any genuine function $f\colon \Zt\rightarrow\C$ is of the form $\epsilon f((z(\lambda),\epsilon))= f(\lambda)$ for $\lambda\in \Fx$ and $\epsilon\in Z_{2}$. We make the convention of taking the normalised counting measure $d\zt$ on $\Zt/\Zt^{2}$. Explicitly, for $f\colon \Zt/\Zt^{2}\rightarrow\C$ define
\begin{equation*}
\int_{\Zt/\Zt^{2}} f(\zt)\,d\zt\, =\,\frac{1}{\#(\Fx/(\Fx)^{2})}\, \sum_{\lambda\in \Fx/(\Fx)^{2}} f(z(\lambda)^{*}).
\end{equation*}

\subsection{Non-archimedean local fields: notations and conventions}\label{sec:non-arch-notation}

In this section we suppose that $F$ is a local field subject to the ultrametric property. 
% Then $F$ is classified by a discrete valuation $v\colon F\rightarrow \Z$.
Denote by $\of$ the ring of integers of $F$ and by $\pf$ the maximal ideal of $\of$. We fix a `uniformiser', that is an $\of$-generator of $\pf$, and denote it by $\varpi$; we let $q=\# (\of/\pf)$. Let $\abs{\,\cdot\,}=\abs{\,\cdot\,}_{F}$ be the absolute value on $F$, normalised so that $\abs{\varpi}=q^{-1}$, and
$v=v_{F}$ the valuation on $F$ defined via $\abs{x} = q^{-v(x)}$.% Define a basis of open neighbourhoods $U_k$ of $1\in\ofx$ by $U_{k}:=\ofx\cap(1+\pf^{k})$ for $k\geq 0$.

For a multiplicative character $\chi$ of $\Fx$ we define its `(exponent) conductor' to be the integer
\begin{equation*}
a(\chi)=\min\{k\geq 0\,: \,\chi(u)=1\text{ for all } u\in \ofx\cap(1+\pf^{k})\}.
\end{equation*}
Once and for all fix an additive character $\psi$ of $F$ which we assume to be unramified; that is, $\min\{r\geq 0 : \psi\vert_{\pf^{r}}=1\}=0$. The additive group $F$ is self-dual. After making our choice of $\psi$, we identify the dual group of $F$ by defining $\psi_{\lambda}(x)=\psi( \lambda x)$ for each $\lambda\in F$.

\subsubsection{Congruence subgroups}

The group $\GL_{2}(F)$ contains the maximal compact subgroup $K:=\GL_{2}(\of)$ and, for each integer $n\geq 0$, the congruence subgroups
\begin{equation*}
K_{0}(n)=\left\lbrace\, \begin{pmatrix} a&b\\c&d\end{pmatrix}\in K\,:\, c\in \pf^{n}\,\right\rbrace\leq K
\end{equation*}
and
\begin{equation*}% \arraycolsep=3pt\def\arraystretch{0.5}
K_{1}(n)=\left\lbrace\,\begin{pmatrix} a&b\\c&d\end{pmatrix}\in K\,:\,c\in \pf^{n},\, d\in U_{n}\,\right\rbrace\leq K_{0}(n),
\end{equation*}
components of the filtrations $K_{i}(n)\subset K_{i}(n-1)\subset\cdots\subset K_{i}(0)=K$ for $i=0,1$. Similarly we define $K^1(n)$ by assuming $a\in U_n$ rather than $d$. Over the metaplectic group, \eqref{eq:exact-sequence} splits trivially over $K^1(n)$ (see \cite[Prop.\ 2.8]{gelbart-spectrum}) as long as $v_p(4)\leq n$. In other words $K_0(n)^*$ (see \eqref{eq:g-star} for notation) defines a subgroup of $\GLt_2(F)$ as long as $F$ has odd residual characteristic.

\subsubsection{Measure considerations}\label{sec:non-arch-measures}

Let $dx$ be the Haar measure on $F$, normalised so that $\Vol(\of,dx)=1$, and $\dxy$ the Haar measure on $\Fx$, normalised so that $\Vol(\ofx,\dxy)=1$. Let $\zeta(s)=(1-q^{-s})^{-1}$ denote the (local) zeta-function of $F$. Note that $\dxy=\zeta(1)\abs{y}^{-1} dy$.

Let us normalise a right-Haar measure $d_{R}b$ on $B=ZNA$ via the formula $$d_{R}b= d^{\times}u\, dx \, \dxy$$ for $b=z(u)n(x)a(y)$ with $u,y\in\Fx$ and $x\in F$. Then $d_{L}b:= \abs{y}^{-1}d_{R}b$ determines a left-Haar measure on $B$. We denote the unique Haar probability measure on $K$ by $dk$. The group $\GL_2$ is unimodular, thus it posesses a bi-invariant Haar measure $dg$. Using the Iwasawa decomposition we can make the identification 
\begin{equation}
dg = dk d_Rb = d_Lbdk. \nonumber
\end{equation}
Note that this identification determines a canonical normalisation of $dg$.

\subsection{Global aspects of the metaplectic group}

To introduce $\GLt_{2}(\A)$ we follow the discussion in \cite[\S 2.2]{gelbart-spectrum}. We define the adele group $\GLt_{2}(\A)= \GL_2(\A)\times Z_2$ and equip it with the product 
\begin{equation}
	(g,\epsilon_1)\cdot (h,\epsilon_2) = (g\cdot h,\epsilon_1\epsilon_2\prod_v\beta_v(g_v,h_v)).\nonumber
\end{equation}
Note that if $g_{v},h_v\in K_v$ and $v$ has odd residual characteristic, then $\beta_v(g_v,h_v)=1$. Thus, by construction of the adeles as a restricted direct product we are always encountering only finite products of co-cycles. This makes our group well-defined.

Another more conceptual construction is given by quotient
$$\GLt_2(\A) = \Zt_{\eve}\backslash {\prod_v}'\GLt_2(F_v),$$
where 
\begin{equation}
	\Zt_{\eve}=\left\lbrace (\epsilon_v)_v\in {\textstyle\prod_v} Z_2\colon \epsilon_v=1 \text{ for all but a finite, even number of places }v\right\rbrace.\nonumber
\end{equation} 

In order to discretely embed $\GL_2(F)$ we have to construct another lifting other than
$
g^{*}=(g,1)\in\GLt_{2}(\A)
$
for $g\in\GL_{2}(\A)$ as given in \eqref{eq:g-star}. We define a function on $\gamma\in\GL_{2}(F)$ by setting $s(\gamma)=\prod_{v}s_{v}(\gamma)$; note that this product is indeed finite since the valuation $v(c)=0$ for almost all $v$. Now define an embedding $\varsigma\colon\GL_{2}(F)\emb\GLt_{2}(\A)$ by mapping $\gamma\in\GL_{2}(F)$ to
\begin{equation}\label{eq:diagonal-embedding}
\gamma^{\varsigma}:=(\gamma,s(\gamma))\in \GLt_{2}(\A)
\end{equation}
such that $\gamma$ is embedded diagonally into $\GL_{2}(\A)$ in the first variable. Then $\varsigma$ is in fact an injective group homomorphism.

\section{Whittaker Models, Kirillov Models and Bessel Transforms}
\label{sec:back}

We now recap some necessary prerequisites concerning local and global Whittaker models and their corresponding Bessel transforms.

\subsection{Local Whittaker and Kirillov models}

We refer to \cite[\S  3]{gelbart-ps-on-shimura}. Let $(\pi,V_{\pi})$ be an irreducible admissible genuine representation of $\GLt_2(F)$. The space $\mathcal{W}(\pi,\psi)$ of smooth functions on $\GLt_2(F)$, invariant under right translation by $\GLt_2(F)$, such that 
\begin{equation}
W(n(x)^* g) = \psi(x)W(g)\nonumber
\end{equation}
for all $x \in F$, $g\in \GLt_2(F)$, and  $W\in \mathcal{W}(\pi,\psi)$ is called a \textit{$\psi$-Whittaker model} for $\pi$ if the representation of $\GLt_2(F)$ given by right translation $R$ on $\mathcal{W}(\pi,\psi)$ is equivalent to $\pi$. Such models always exist. However, in contrast to the $\GL_2$ setting they are not necessarily unique. Therefore we call $\pi$ `distinguished' if $\mathcal{W}(\pi,\psi)$ is unique.  

To fix this lack of uniqueness we make the following definitions. A $(\psi,\mu)$-Whittaker functional on $\pi$ is a linear map $l^{\mu}\colon V_{\pi}\to\C$ such that
\begin{equation}
l^{\mu}(\pi(\zt n(x)^*)\xi) = \mu(\zt)\psi(x)l^{\mu}(\xi)\nonumber
\end{equation}
for all $x\in F$, $\zt\in \Zt$, and $\xi\in V_{\pi}$. Such functionals \textit{are} unique and for each representation $\pi$ there exists at least one $\mu\in \Omega(\omega_{\pi})$ such that $l^{\mu}\neq 0$. 
Thus, we define the non empty set $\Omega(\pi) = \{\mu\in\Omega(\omega_{\pi})\vert l^{\mu}\neq 0 \}$. Given such a $l^{\mu}$ we can construct the unique $(\psi,\mu)$-Whittaker model by defining
\begin{equation}
\mathcal{W}(\pi,\psi,\mu) = \{ W_{\xi}^{\mu}\colon g\mapsto l^{\mu}(\pi(g)\xi) \colon \xi\in V_{\pi} \}. \nonumber
\end{equation} 
In particular the elements $W\in \mathcal{W}(\pi,\psi,\mu)$ are smooth and satisfy
\begin{equation}
W(\zt n(x)^*g) = \mu(\zt)\psi(x)W(g) \nonumber
\end{equation}
for all $x\in F$, $\zt \in \Zt$, and $g\in\GLt_2(F)$. The representation of $\GLt_2(F)$ on $\mathcal{W}(\pi,\psi,\mu)$ is equivalent to $\pi$.

If $F$ is non-archimedean we have the following useful result.

\begin{lemmy}\label{lem:support-whittaker}
	On the group $A^{*}$, the Whittaker functions $W\in \mathcal{W}(\pi,\psi,\mu)$, which are right-invariant by $N(\of)^*$ satisfy $W(a(y)^{*})=0$ for all $y\in \Fx$ with $\abs{y}>1$.
\end{lemmy}
\begin{proof}
	Let $y\in F^{\times}$. Then for all $x\in\of$ we have
	\begin{equation}
	W(a(y)^{*}) = W(a(y)^{*}n(x)^*) = W(n(xy)^*a(y)^{*}) = \psi(xy)W(a(y)^{*}). \nonumber
	\end{equation}
	If $\abs{y}<1$, we can always find $x\in \of$ such that $\psi(xy) \neq 1$.
\end{proof}

Let $\mathcal{C}_0(F^{\times})$ be the image of the map $f\mapsto f\vert_{F^{\times}}$ for Schwartz functions $f\in\mathcal{S}(F)$. Then the `Kirillov map' is defined by
\begin{equation}
	\xi \mapsto (W_{\xi}^{\mu}(a(\cdot)^*))_{\mu\in \Omega(\pi)} \in \mathcal{C}_0(F^{\times})^{\sharp \Omega(\pi)}. \nonumber
\end{equation}
for $\xi\in V_{\pi}$.
As explained in \cite[\S 3]{gelbart-ps-distinguished} this map is injective. Therefore, its image is equivalent to $V_{\pi}$ and is called the `$\psi$-Kirillov model' of  $\pi$. We write $K(\pi,\psi)$ for this space. Moreover, one can show that it contains the subspace $K_0(\pi,\psi)=\oplus_{\mu\in\Omega(\pi)} \mathcal{S}(F^{\times})$.
	
For each $\mu\in \Omega(\pi)$ we can project onto the $\mu$-fibre of the $(\pi,\psi)$-Kirillov model. This leads to the function space $\{y\mapsto  W(a(y)^{*})\,:\, W\in \Wh(\pi,\psi,\mu)\}$, which contains the Schwartz--Bruhat functions $\Sch(\Fx)$. We refer to the latter space as `$(\psi,\mu)$-Kirillov space of $\pi$'. Note that the $(\psi,\mu)$-Kirillov space is not a proper model for $\pi$.

\subsection{Archimedean Bessel transforms}

In this section we introduce the Bessel transform for genuine representations of $\GLt_{2}(F)$. The goal is to describe the action of $w$ in the Kirillov model as an integral transform. We outline this theory following \cite{baruch-mao-bessel-p-adic,baruch-mao-bessel-real}.

\subsubsection{The real Bessel transform}\label{sec:bessel-real}

First consider $F=\R$. We start by quickly recalling the classification of representations in this case. 

\begin{prop}[Classification of representations of $\GLt_{2}(\R)$]\label{prop:classification-real}
Every irreducible, admissible, genuine, unitary representation $\pi$ of $\GLt_{2}(\R)$ is equivalent to
\begin{equation*}
	\pi(\mu,\sigma):=\ind_{\Zt\SLt_2(\R)}^{\GLt_2(\R)}(\mu\times\sigma)
\end{equation*}
for some genuine character $\mu$ of $\tilde{Z}$ and some genuine irreducible admissible representation $\sigma$ of $\SLt_2(\R)$ satisfying $\sigma\vert_{\Mt} = \mu\vert_{\Mt}$. We distinguish several cases:
\begin{enumerate}	
	\item The Principal (or Continuous) Series representations. Let $\sigma = \sigma_{\alpha,s}$ be the representation induced from
	\begin{equation}
		\Mt N^*(A_0^+)^*\ni  m^j n(x)^*a'(y)^* \mapsto y^{s+1}e(\frac{\alpha j}{4}) \in \C, \nonumber
	\end{equation}
	for $\alpha = 1,3$ and $s\in i\R_{>0}$. We write $$\pi(\mu,s)=\pi(\mu,\sigma_{1,s})=\pi(\sgn\cdot\mu,\sigma_{3,s}).$$ Note that
	\begin{equation}
		\pi(\mu,s)\vert_{\Zt\SLt_{2}(\R)} =  \mu\times \sigma_{1,s}\oplus (\sgn\cdot\mu) \times \sigma_{3,s}. \nonumber
	\end{equation} 
	The action of $a(-1)^*$ is easily determined by permuting the two components. Further we must have $\mu(z(-1)^*)=i$.
		
	\item The Complementary Series representations. These are given by $\sigma=\sigma_{\alpha,s}$ for $s\in\R$ with $0<s<1/2$ and $\alpha\in\{1,3\}$. As in the Principal series case we  write $\pi(\mu,s)=\pi(\mu,\sigma_{1,s})=\pi(\sgn\cdot\mu,\sigma_{3,s})$.
		
	\item The Discrete Series representations. The representation $\sigma_{1,k-\frac{1}{2}}$ contains a unitary irreducible subspace $\tilde{\sigma}_{1,k-\frac{1}{2}}$. If $k$ is even this is a holomorphic discrete series and if $k$ is odd it is a anti-holomorphic discrete series. A similar statement with holomorphic and anti-holomorphic reversed holds for $\sigma_{3,k-\frac{1}{2}}.$ We write $\sigma(\mu,k) = \pi(\mu,\tilde{\sigma}_{1,k-\frac{1}{2}}) = \pi(\sgn\cdot\mu,\tilde{\sigma}_{3,k-\frac{1}{2}})$. We can always arrange that
	\begin{equation}
		\sigma(\mu,k)\vert_{\SLt_{2}(\R)} = \begin{cases}
			\tilde{\sigma}_{1,k-\frac{1}{2}}\oplus \tilde{\sigma}_{3,k-\frac{1}{2}}&\text{ if $k$ is even,}\\
			\tilde{\sigma}_{3,k-\frac{1}{2}}\oplus \tilde{\sigma}_{1,k-\frac{1}{2}}&\text{ if $k$ is odd.} 	
		\end{cases}\nonumber
	\end{equation}
	The point is that the first component is always the holomorphic discrete series of lowest weight $k+\frac{1}{2}$. We may view this discrete series representation as a subspace of $\pi(\mu,s)$ for $s=k-\frac{1}{2}$.
		
	\item The Weil representation. If $\sigma$ is the irreducible subspace of $\sigma_{\alpha,-\frac{1}{2}}$, then $\pi(\mu,\sigma)$ is determined by a character $\chi$ of $\Rx$ and known as the $r_{\chi}^{+}$-Weil representation.
	\end{enumerate}
\end{prop}

Over $\R$ we have to consider $(\psi,\mu^{\pm})$-Whittaker models. Where $\mu^{\pm} = \sgn^{\frac{1\pm 1}{2}}\mu$. At least one of them will be non-zero. In particular we have the Whittaker models $\Wh(\pi(\mu,s),\psi,\mu^{\pm})$. In order to fix two standard Whittaker functionals we model $\pi(\mu,s)$ on
\begin{multline}
	\mathcal{F}_{\pi} =\bigg\{ F\in\mathcal{C}^{\infty}(\GLt_2^+(\R)\times \GLt_2(\R),\C)\colon  [h\mapsto F(h,g)]\in \mu\times\sigma \text{ for all }g\in\GLt_2(\R) \\ \text{ and }F(h_1h_2,g) = F(h_1,h_2g) \text{ for all }h_1,h_2\in\GLt_2^+(\R), \ g\in\GLt_2(\R) \bigg\}.\nonumber
\end{multline}
Set $\psi_{\lambda}(x) = e(\lambda x).$ The standard $(\psi_{\lambda},\mu^{\pm})$-Whittaker functionals (agreeing with those defined in \cite{baruch-mao-bessel-real}[\S 13]) are given by
\begin{equation}
	L^{\pm}_{\lambda}(F) = \int_{\R} F(w^*n(x)^*,a(\pm 1)^*)e(\mp\lambda p x)\abs{\lambda}^{\frac{1}{2}}dx.\nonumber
\end{equation}

We now turn towards the action of the element $w^*$ on the Kirillov model. This action is given by several Bessel transform.

\begin{prop}\label{prop:bessel-real}
Let $\pi= \pi(\mu,\sigma)$ be a genuine, irreducible representation of $\GLt_{2}(\R)$ as in Proposition \ref{prop:classification-real}. For $\epsilon\in \{\pm 1 \}$ there exist functions $j_{\pi}^{\epsilon,\pm}\colon \Rx\rightarrow\C$ such that
\begin{equation}
	W_v^{\mu^{\epsilon}}(a(x)^*w^*) = \sum_{\pm} \int_{\R^{\times}}[-x,y]\mu^{\epsilon}(y^{-1})j_{\pi}^{\epsilon,\pm}(xy)W_v^{\mu^{\pm}}(a(y)^*)d^{\times}y\nonumber
\end{equation}
whenever $W_v^{\mu^{\pm}}(a(\cdot)^*)\in \mathcal{S}(\R^{\times})$.
\end{prop}
\begin{proof}
	This follows from \cite[Theorem~13.2]{baruch-mao-bessel-real} and the identification $$[-x,y]\mu^{\epsilon}(y^{-1})j_{\pi}^{\mu^{\epsilon},\mu'}(xy) = [-x,y]\mu^{\epsilon}(y^{-1})j_{\pi}^{\mu^{\epsilon},\mu'}(a(xy)^*w^*) = j_{\pi}^{\mu^{\epsilon},\mu'}(a(y)^*w^*a(x^{-1})^*).$$
\end{proof}
Note that in contrast to the $\GL_2$ situation we obtain a linear combination of two Bessel transforms. This is a reflection of the non-uniqueness of the Whittaker model and can be directly seen from the more complicated structure of the Kirillov model. Moreover, if $\mu\neq \mu'$, then the Bessel function $j_{\pi}^{\mu,\mu'}$ depends on the normalisation of the two Whittaker functionals $L^{\pm}$. Otherwise the Bessel function depends only on the measure normalisations.

We finish this section by listing explicit formulae for these functions in some cases. We will consider the additive character $\psi(x) = \psi_{\lambda}(x)= e(\lambda x)$ and write $j_{\pi,\lambda}^{\epsilon,\pm}(x)$ to highlight the additional dependence. The following expressions are extracted from \cite[p.53]{baruch-mao-bessel-real}.

For \textit{discrete series representations}, $s \in -\frac{1}{2}+\N$, we have $j_{\pi,\lambda}^{\pm,\mp}(x)=0$ and $j_{\pi,\lambda}^{-,-}(x)=j_{\pi,-\lambda}^{+,+}(x)$. Therefore, it is enough to describe the $++$-case. %If we are in the \textbf{holomorphic} case, $s\in \frac{1}{2}(-1)^{\eta}+2\N$, 
We have 
\begin{equation}
j_{\pi,\lambda}^{+,+}(x)= \begin{cases}
2\frac{\pi e^{-3\pi i s/2}}{\sin(\pi s)}\mu(\sqrt{\abs{x}})\abs{x\lambda}^{\frac{1}{2}}J_s(4\pi\abs{\lambda}\sqrt{\abs{x}}) &\text{ if }x,\lambda>0, \\	
0 &\text{ else.}
\end{cases} \label{eq:holo_bessel}
\end{equation}
%In the \textbf{anti-holomorphic} case, $s\in -\frac{1}{2}(-1)^{\eta}+2\N$, we have
%\begin{equation}
%	j_{\pi,\lambda}^{+,+}(x)= \begin{cases}
%		-2\frac{\pi e^{-3\pi i s/2}}{\sin(\pi s)}\mu(\sqrt{\abs{x}})\abs{x\lambda}^{\frac{1}{2}}J_s(4\pi\abs{\lambda}\sqrt{\abs{x}}) &\text{ if } \lambda<0<x, \\	
%		0 &\text{ else.}
%	\end{cases} \nonumber
%\end{equation}

For \textit{principal series representations}, $s\in i \R$, and \textit{complementary series}, $0<s<\frac{1}{2}$, we have
\begin{align}
j_{\pi,\lambda}^{+,+}(x) &= 
\delta_{x>0}2\mu(\abs{x}^{\frac{1}{2}})\abs{\lambda x}^{\frac{1}{2}} [K_s(\lambda\cdot 4\pi i \abs{x}^{\frac{1}{2}})-iK_s(-\lambda\cdot 4\pi i \abs{x}^{\frac{1}{2}})], \nonumber \\
j_{\pi,\lambda}^{+,-}(x) &= 
\delta_{x>0}2\mu(\abs{x}^{\frac{1}{2}})\abs{\lambda x}^{\frac{1}{2}}[(\sgn(\lambda)i)^{-s-1}-i(-\sgn(\lambda)i)^{-s-1}]K_s(4\pi\abs{\lambda^2 x}^{\frac{1}{2}}) \nonumber
\end{align}
where we consider the situation $\pi=\pi(\mu,\sigma_{1,s})$. In particular we have $\mu(-1)=i$. The other Bessel transforms are obtained by the relations $j_{\pi,\lambda}^{-,-}(x) = j_{\pi,-\lambda}^{+,+}(x)$ and $j_{\pi,\lambda}^{-,+}(x)=j_{\pi,-\lambda}^{+,-}(x)$.

\begin{rem}
We skip the case of the even Weil representation, as this can be easily excluded in most global applications.
\end{rem}

\subsubsection{The complex Bessel transform}\label{sec:bessel-complex}

If $F=\C$, we are in the split situation where $\GLt_{2}(\C) = \GL_2(\C)\times Z_2$. Thus, every genuine, irreducible, unitary representation is given by
\begin{equation}\label{eq:complex-representation}
\tilde{\pi}((g,\epsilon))v=\epsilon \pi(g)v, \nonumber
\end{equation}
where $\pi$ is an irreducible, unitary representation of $\GL_2(\C)$ and $v\in\pi$.  The latter representations have been classified in \cite[Theorem~6.2]{j-l-gl2}. Also the complex Bessel transform reduces to the $\GL_2(\C)$ case, which has been studied in \cite{qi-fundamental-theory-bessel}.

\subsection{Bessel transforms over non-archimedean fields}

Let $F$ denote a non-archimedean local field and fix an irreducible, admissible representation $(\pi,V_{\pi})$ of $\GLt_2(F)$. For $\mu\in \Omega(\omega_{\pi})$, fix a $(\psi,\mu)$-Whittaker functional for $\pi$ and call it $l^{\mu}$. By \cite[Lemma 7.1]{baruch-mao-bessel-p-adic}, the $(\psi,\mu)$-Whittaker space may be decomposed as
\begin{equation}
\Wh(\pi,\psi,\mu) = \bigoplus_{\mu'\in\Omega(\omega_{\pi})} \Wh_{\mu'}(\pi,\psi,\mu), \nonumber 
\end{equation}
where we define the space $\Wh_{\mu'}(\pi,\psi,\mu)$ to be the set of functions $W\in \Wh(\pi,\psi,\mu)$ satisfying $W(g\zt)= \mu'(\zt)W(g)$ for all $g\in\GLt_2(F)$ and $\zt\in\Zt$.

For $W\in \Wh(\pi,\psi, \mu)$ and $\mu'\in\Omega(\omega_{\pi})$ we define the function
\begin{equation} 
J_{n}(W,\mu';g) :=  \, \int_{\pf^{-n}}\int_{\Zt/\Zt^{2}} \,W(gn(x)^*\zt)\,\psi(x)^{-1}\,\mu'(\zt)^{-1}\,d\zt\, dx. \nonumber
\end{equation}
By \cite[Theorem 8.2]{baruch-mao-bessel-p-adic} the sequence stabilises as $n\to \infty$. Thus, 
\begin{equation}\label{eq:bessel-J-def}
J(W,\mu';g) = \lim_{n\to \infty} J_{n}(W,\mu';g)
\end{equation}
converges for each $g\in \Bt w^{*}N^{*}$, is locally constant on this set, and defines a locally integrable function on $\GLt_2(F)$. 

Note that  by construction $J(W_{\xi}^{\mu},\mu';g)$ gives rise to a $(\psi,\mu')$-Whittaker functional on $\GLt_2(F)$. By the uniqueness property for such functionals we conclude that there exists a value $j_{\pi}^{\mu,\mu'}(g)\in\C$ such that
\begin{equation}\label{eq:bessel-little-j-def}
J(W_{\xi}^{\mu},\mu';g) = j_{\pi}^{\mu,\mu'}(g)\, W_\xi^{\mu'}(1)
\end{equation}
for each $g\in\Bt w^*N^*$. The consequent function $ j_{\pi,\lambda}^{\mu,\mu'}(g)$ also depends on the fixed choice of character $\psi$; Whittaker functional $l^{\mu'}$; and Haar measure $dx$ present in \eqref{eq:bessel-J-def}, though is independent of the vector $\xi$. If $l^{\mu'}=0$ simply define $j_{\pi}^{\mu,\mu'}(g)=0$ identically. 

We call the function $j_{\pi}^{\mu,\mu'}$ the `$(\mu,\mu')$-Bessel function' associated to $\pi$. We are interested in these functions because they can be used to evaluate the Whittaker functions $W_{\xi}^{\mu}$ away from the diagonal. By \cite[Proposition~8.3]{baruch-mao-bessel-p-adic} the Bessel transform is given as follows (noting that a factor of $\zeta_F(1)^{-1}$ appears because our measure is normalised differently).

\begin{prop}\label{prop:bessel-padic}
	Let $\xi\in V_{\pi}$ such that for all $\mu'\in \Omega(\omega_{\pi})$ the function $y\mapsto W_\xi^{\mu'}(a(y)^{*})$ has compact support in $\Fx$. Then for each $g\in\Bt w^*N^*$ we have
	\begin{equation}
	W_\xi^{\mu}(g) = \zeta_F(1)^{-1}\sum_{\mu'\in\Omega(\omega_{\pi})}\int_{F^{\times}}j_{\pi}^{\mu,\mu'}(ga(y^{-1})^*)\,W_{\xi}^{\mu'}(a(y)^{*})\,\dxy.\nonumber
	\end{equation}
\end{prop}

The left- and right-translates by $\Zt N^{*}$ of the $(\mu,\mu')$-Bessel function are easily described. Thus, by the Bruhat decomposition for $\GL_2(F)$, it suffices to understand the Bessel functions on elements $(a(y),\epsilon)w^{*}\in\Gt$ for $y\in\Fx$ and $\epsilon\in Z_{2}$. Abusing notation, we define
\begin{equation}
j_{\pi}^{\mu,\mu'}(y,\epsilon) := j_{\pi}^{\mu,\mu'}((a(y),\epsilon)w^*); \quad j_{\pi}^{\mu,\mu'}(y) := j_{\pi}^{\mu,\mu'}(y,1). \nonumber
\end{equation}
These functions immediately satisfy the following support condition.

\begin{lemmy}\label{lem:bessel-supp-cond}
If $y\in \Fx$ then $j_{\pi}^{\mu,\mu'}(y)=0$ unless $\mu'(u) = \,[u,y]\mu(u)$ for all $u\in\Fx$.
\end{lemmy}
\begin{proof}
By \eqref{eq:bessel-little-j-def} it suffices to consider the function of $J(W,\mu';w^{*}a(y)^{*})$ for some $W\in\Wh(\pi,\psi,\mu)$. By reordering integration in \eqref{eq:bessel-J-def}, for each $u\in\Fx$ we obtain
\begin{equation}
	j_{\pi}^{\mu,\mu'}(w^{*}a(y)^{*})\mu'(u) =  [u,y]\, \mu(u)\, j_{\pi}^{\mu,\mu'}(w^{*}a(y)^{*}), \nonumber
\end{equation}
whence the claim follows.
\end{proof}

\begin{rem}
If $\pi$ is distinguished, then there is a unique (genuine) character $\mu_{\pi}$ satisfying $l^{\mu_{\pi}} \neq 0$. Therefore, the only non-trivial Bessel function is $j_{\pi}^{\mu_{\pi},\mu_{\pi}}$. This function does not depend on the choice of $l^{\mu_{\chi}}$ but only upon the choice of Haar measures in \eqref{eq:bessel-J-def}. An explicit formula for this Bessel function is given in \cite[Proposition~4.4.2]{gelbart-ps-distinguished}. Even more interesting is that for such distinguished representations there is a local functional equation, which closely mimics the $\GL_2(F)$ analogue. Indeed, there exist constants $\gamma(s,\pi,\psi)\in\Cx$, for $s\in \C$, known explicitly in terms of local $L$- and $\varepsilon$-factors, which are shown to equal
\begin{equation}\label{eq:distinguished-gamma}
\gamma(s,\pi,\psi) = \int_{F^{\times}}j_{\pi}^{\mu_{\pi},\mu_{\pi}}(y)\mu_{\pi}(y)^{-1}\abs{y}^{\frac{1}{2}-s}d^{\times}y.
\end{equation}
(See \cite[p.~170]{gelbart-ps-distinguished}.) 

For non-distinguished representations $\pi$ such a compact formula relating the Bessel functions $j_{\pi}^{\mu,\mu'}$ to some $\gamma$-factors is not available. However, adapting the proof of \cite[Lemma~4.5]{soudry-the-L-and-gamma} one arrives at
\begin{eqnarray}
\gamma(s,\pi,\pi',\psi,\mu) &=& \zeta_F(1)^{-2}\int_{F^{\times}}\psi(u)\mu\mu_{\pi'}(u)^{-1}\abs{u}^{2(1-s)}\nonumber \\
&&\qquad\qquad\cdot \int_{u^{-2}\pf^{-n}}j_{\pi}^{\mu,\mu}(y)j_{\pi'}^{\mu_{\pi'},\mu_{\pi'}}(y)\mu\mu_{\pi'}(y)^{-1}\abs{y}^{-s} d^{\times}yd^{\times}u,\nonumber
\end{eqnarray}
for $n\gg 1$ and a distinguished representation $\pi'$. Here the $\gamma$-factors appearing on the right come from a local functional equation of \textit{Shimura-type}, see \cite[Theorem 5.3]{gelbart-ps-on-shimura}, and play a key role in the representation theoretic description of the Shimura correspondence.

\end{rem}

\subsection{Global aspects}

Let $\pi$ be a genuine cuspidal automorphic representation of $\GLt_2(\A)$. These are irreducible constituents of the space of genuine cuspidal automorphic forms $\mathcal{A}^{\circ}(\GLt_2(\A))$ on which $\GLt_2(\A)$ acts by right translation; see \cite{gelbart-ps-on-shimura}[\S 9.3]. We write $\omega_{\pi}$ for its central character. It can be shown that each genuine irreducible admissible representation $\pi$ of $\GLt_2(\A)$ has a factorisation
\begin{equation}
	\pi = \bigotimes_v \pi_v, \nonumber
\end{equation}
where $\pi_v$ are genuine irreducible admissible representations of $\GLt_2(F_v)$, which are spherical for almost all $v$. Note that the tensor product is actually a representation of ${\prod_v}'\GLt_2(F_v)$ but, since all the local representations are genuine, the action of $\Zt_{\eve}$ is trivial so that the product descends to a representation of $\GLt_2(\A)$ viewed as a quotient. 

Let $\psi=\otimes_v \psi_v$ be an additive character of $\A$ which is invariant under $F$. We further assume that all the local components $\psi_v$ are unramified. Using Fourier analysis on the compact quotient $F\backslash \A$, one finds that any genuine cuspidal automorphic form $\phi$ has a Fourier, or Whittaker, expansion
\begin{equation}
	\phi(g) = \sum_{\xi\in F^{\times}}W_{\phi}(a(\xi)^*g) \label{eq:whitt_exp}
\end{equation}
for $$W_{\phi}(g) = \int_{F\backslash \A} \phi(n(x)^*g)\overline{\psi(x)}dx.$$ This is just as in the $\GL_2$ situation. For obvious reasons we call $W_{\phi}$ the $\psi$-Whittaker function of $\phi$. Guided by the local situation we define the set $\Omega(\omega_{\pi})$ as the set of genuine characters $\mu\colon \Zt(\A)\to\Cx$ such that $\mu^2=\omega_{\pi}$. Furthermore we define the $(\psi,\mu)$-Whittaker function of $\phi$ by
\begin{equation}
	W_{\phi}^{\mu}(g) = \int_{\Zt(\A)^2\backslash \Zt(\A)}\int_{F\backslash \A} \phi(zn(x)^*g)\overline{\psi(x)\mu(z)}dxdz.\nonumber
\end{equation}
Of course we have
\begin{equation}
	W_{\phi}(g) = \sum_{\mu\in\Omega(\omega_{\pi})}W_{\phi}^{\mu}(g).\nonumber
\end{equation}
We write $\mu^{\delta}(z) = \mu((a(\delta)^{-1})^*za(\delta)^*)$ and obtain the refined Whittaker expansion
\begin{equation}
	\phi(g) = \sum_{\mu\in\Omega(\omega_{\pi})}\sum_{\xi\in F^{\times}}W_{\phi}^{\mu^{\xi}}(a(\xi)^*g). \label{eq:ref_whitt_exp}
\end{equation}

As we did in the local case, we can associate the global $(\psi,\mu)$-Whittaker model $\Wh(\pi,\psi,\mu)$ to $\pi$. We define $\Omega(\pi)\subset\Omega(\omega_{\pi})$ to consist of those $\mu$ for which $\Wh(\pi,\psi,\mu)$ exists and is non-trivial. We call $\pi$ distinguished if $\sharp\Omega(\pi)=1$.

\begin{rem}
Let $\pi =\otimes_v \pi_v$ be a genuine cuspidal automorphic representation of $\GLt_{2}(\A)$. Then the following are equivalent:
\begin{itemize}
	\item $\pi$ is distinguished.
	\item $\pi$ is a theta representation. In other words, it is constructed from a global Weil representation.
	\item $\pi_v$ is distinguished for every $v$.
\end{itemize} 
This is the content of \cite[Theorem A]{gelbart-ps-on-shimura}. Furthermore if $\phi$ transforms with respect to a distinguished cuspidal automorphic  representation $\pi$ with $\Omega(\pi)=\{\mu\}$, then \eqref{eq:ref_whitt_exp} reads
\begin{equation}
	\phi(g)=\sum_{\xi\in F^{\times}}W_{\phi}^{\mu^{\xi}}(a(\xi)^*g) \nonumber
\end{equation}
and we have $W_{\phi}(g)=W_{\phi}^{\mu}(g)$ for all $g$.
\end{rem}

We have now gathered all the necessary representation theoretic background and turn towards the derivation of our main Vorono\"i formula.

\section{The main adelic theorem}
\label{sec:main}

As suggested in the introduction, we follow the strategy for $\GL_2$ as outlined in \S \ref{sec:gltwo-revisited}. For distinguished (cuspidal) automorphic representations $\pi$ this can be executed verbatim. However, such a result is not so interesting as it is essentially just a reformulation of Poisson summation. Therefore our focus will be mainly on non-distinguished representations. For such representations, the fact that $N^*\subset \GLt_2$ is not a strong Gelfand pair becomes apparent. We instead consider the following Gelfand formation:
\[
\begin{tikzcd}
&\GLt_2 \arrow[dash]{rd} \arrow[dash]{ld} &  \\
\Zt N^* \arrow[dash]{rd} & & \Zt\overline{N}^* \arrow[dash]{ld}\\
& \{ 1 \}^* & 
\end{tikzcd}
\]
Note that $\Zt N^*\subset \GLt_2$ is a strong Gelfand pair by uniqueness of $(\psi,\mu)$-Whittaker functionals. Furthermore we have $w^*\Zt N^* (w^*)^{-1} = \Zt\overline{N}$. This follows from the elementary computation
\begin{equation}
	w^* z n(x) (w^*)^{-1} = z\overline{n}(-x), \text{ for }z\in\Zt \text{ and } x\in \A. \nonumber
\end{equation}
Note that
\begin{equation*}
w^{\varsigma}=(w, 1)\in \GL_{2}(F)^{\varsigma}\subset\GLt_{2}(\A).
\end{equation*}
We now give the fundamental identity by expanding the period $\phi\mapsto \phi(e)$ through both sides of the formation.
Equivalently, the fundamental identity may be viewed as the Fourier expansion of both sides of the equality $\phi(1) = \phi(w^{\iota})$, which holds for example for automorphic forms $\phi$. This is the point of view we adopt in the following proof.

\begin{prop}\label{prop:fundamental-identity}
Let $\phi$ be a cuspidal automorphic form transforming with respect to some cuspidal automorphic representation  $\pi=\otimes_{v}\pi_{v}$ of $\GLt_{2}(\A)$. Further assume that $\phi$ is right-$N(\hat{\of})^{*}$-invariant. Let $\zeta=(\zeta_{v})\in \A$ and let $S$ denote the set of places containing the archimedean places $v$ for which $\zeta_{v}\neq 0$ and the non-archimedean places $v$ for which $\abs{\zeta_{v}}_{v}>1$.
Let $\psi=\otimes_{v}\psi_{v}$ be an additive character of $\A/F$.
Then we have
\begin{equation*}
	\sum_{\mu\in\Omega(\omega_{\pi})}\sum_{\xi\in\Fx}   W_{\varphi}^{\mu^{\xi}}(a(\xi)^{*})\, \psi(\xi\zeta)=  \sum_{\mu\in\Omega(\omega_{\pi})}\sum_{\xi\in\Fx}  W_{\phi}^{\mu^{\xi}}(a(\xi)^{*}w^{\varsigma} h(\zeta))\,\prod_{v\in S}
	\psi_{v}(-\xi\zeta_{v}^{-1})
\end{equation*}
where $h(\zeta)=(h_{v})\in \GLt_{2}(\A)$ is given by
\begin{equation*}
	h_{v}=\begin{pmatrix}
	1&\zeta_{v}\\
	-\zeta_{v}^{-1}&0
	\end{pmatrix}^{*}
\end{equation*}
for $v\in S$ and $h_{v}=1$ for $v\not\in S$.
\end{prop}
\begin{proof}
Define $\phi^{\zeta}:=R(n(\zeta)^{*})\phi$ and consider its Whittaker expansion. On the one hand,
\begin{equation*}
	\begin{array}{rcl}\vspace{0.1in}
	\varphi^{\zeta}(1)&=&\displaystyle\sum_{\mu\in\Omega(\omega_{\pi})}\sum_{\xi\in\Fx}  W_{\varphi}^{\mu^{\xi}}(a(\xi)^*n(\zeta)^{*})\\
	&=&\displaystyle\sum_{\mu\in\Omega(\omega_{\pi})}\sum_{\xi\in\Fx} \psi(\xi\zeta)\ W_{\varphi}^{\mu^{\xi}}(a(\xi)^{*})
	\end{array}
\end{equation*}
since $a(\xi)^*n(\zeta)^{*}=n(\xi\zeta)^{*}a(\xi)^{*}$. On the other hand, by the automorphy of $\phi^{\zeta}$, we have
\begin{equation*}
	\phi^{\zeta}(1)=\phi^{\zeta}(w^{\varsigma})=\sum_{\mu\in\Omega(\omega_{\pi})}\sum_{\xi\in\Fx}  W_{\varphi}^{\mu^{\xi}}(a(\xi)^{*}w^{\varsigma}n(\zeta)^{*}).
\end{equation*}
The lemma is now proved by inspecting the argument place-by-place. For instance, if $v\not\in S$ then $\pi_{v}(n(\zeta_{v})^{*})=1$. Otherwise note that
\begin{multline}
	n(\xi\zeta_v^{-1})^*a(\xi)^{*}w^{\varsigma}n(\zeta_{v})^* = a(\xi)^*n(\zeta_v^{-1})^* w^{*} n(\zeta_v)^* \\ = a(\xi)^*w^{\varsigma}\overline{n}(-\zeta_v^{-1})^*n(\zeta_v)^* = a(\xi)^*w^{\varsigma}\begin{pmatrix}
	1&\zeta_{v}\\
	-\zeta_{v}^{-1}&0
	\end{pmatrix}^{*}. \nonumber
\end{multline}
\end{proof}

We now come to our main adelic theorem. Up to choosing $\phi$ and thus $W_{\phi}^{\mu}$ (in other words Step~1 of our model strategy) this is a very general summation formula to accommodate various application settings.

\begin{theorem}\label{thm:adelic-Voronoi}
Let $\phi$ be a cuspidal automorphic form  with respect to some cuspidal automorphic representation  $\pi=\otimes_{v}\pi_{v}$ of $\GLt_{2}(\A)$. Further assume that $\phi$ is right-$N(\hat{\of})^{*}$-invariant and that it corresponds to a pure tensor. Let $\zeta=(\zeta_{v})\in \A$ and fix the following sets of places: let $S_{\zeta}$ be the set of non-archimedean places $v$ for which $\abs{\zeta_{v}}_{v}>1$ and at which $\phi_v$ is spherical; let $S_{\phi}$ be the set of places $v$ for which $\phi_{v}$ is not spherical together with all the archimedean places. Let $\psi=\otimes_{v}\psi_{v}$ be an additive character of $\A/F$.
Then we have
\begin{equation*}\begin{array}{l}
	\displaystyle\sum_{\mu\in\Omega(\omega_{\pi})}\sum_{\xi\in\Fx}   W_{\phi}^{\mu^{\xi}}(a(\xi)^{*})\, \psi(\delta\zeta)\,=\\\displaystyle \quad \sum_{\mu\in\Omega(\omega_{\pi})}\sum_{\xi\in\Fx} 
	[\mathcal{H}_{S_{\phi}}^{\mu^{\xi}}W_{\phi}^{\mu^{\xi}}](a(\xi)^*\iota_{S_{\zeta}}(a(\zeta^{-2}))^*) \,\prod_{v\in S_{\zeta}}[\zeta_v,-1]\mu_v(\zeta_v)\psi_{v}(-\xi\zeta_{v}^{-1})
\end{array}
\end{equation*}
where $\iota_{S_{\zeta}}$ denotes the map from $ \prod_{v\in S_{\zeta}}G(F_v)\to \GL_2(\A_{F})$, and
\begin{equation*}
	[\mathcal{H}_{S_{\phi}}^{\mu^{\xi}}W_{\phi}^{\mu^{\xi}}](g) = \prod_{v\notin S_{\phi}}W_{\phi,v}^{\mu^{\xi}}(g_v) \cdot \prod_{v\in S_{\varphi}} \sum_{\mu_v'\in \Omega(\omega_{\pi,v})} [\mathcal{H}_{v}^{\mu_v^{\xi},\mu_v'}W_{\phi,v}^{\mu_v^{\xi}}](g_v), 
\end{equation*}
for local transforms given by
\begin{multline}
	[\mathcal{H}_{v}^{\mu_v^{\xi},\mu_v'}W_{\phi,v}^{\mu_v^{\xi}}]((a(\xi)^*) =
	\\
	\zeta_v(1)^{-\delta_{v\nmid\infty}}\int_{F_v^{\times}}\psi_v(y\zeta_v)W_{\phi,v}^{\mu_v'}(a(y)^{*}) j_{\pi_v}^{\mu_v^{\xi},\mu_v'}(\xi y)[y,-1]\mu_v(y)^{-1}d^{\times}y. \nonumber
\end{multline}
\end{theorem}
\begin{proof}
This is a combination of Proposition \ref{prop:fundamental-identity} alongside evaluating $W_{\phi}$ place-by-place according to the local Bessel transform results (Propositions \ref{prop:bessel-real} \& \ref{prop:bessel-padic} and \S~\ref{sec:bessel-complex}) which we now detail.

Let $h(\zeta)$ be as in Proposition \ref{prop:fundamental-identity}.
Starting from Proposition \ref{prop:fundamental-identity}, it remains to evaluate
\begin{equation}
	W_{\phi}^{\mu^{\xi}}(a(\xi)^{\varsigma}w^{\varsigma} h(\zeta)) = \prod_v W_{\phi,v}^{\mu_v^{\xi}}(a(\xi)^{\varsigma}w^{\varsigma} h_v).\nonumber
\end{equation}
We do so by considering several cases.

First, we treat $v\notin S_{\phi}\cup S_{\zeta}$. In this case we simply observe that
\begin{equation}
	W_{\phi,v}^{\mu_v^{\xi}}(a(\xi)^{\varsigma}w^{\varsigma} h_v) = W_{\varphi,v}^{\mu^{\xi}}(a(\xi)^{*}), \nonumber
\end{equation}
since $h_v=1$ and $w^{\varsigma}\in K_v^*$.
	
Second, we consider $v\in S_{\zeta}$. We have the decomposition 
\begin{equation}
	h_v = (z(\zeta_v),[\zeta_v,-1])\overline{a}(\zeta_v^{-2})^*\left(\begin{matrix} \zeta_v^{-1} & 1 \\ -1 & 0\end{matrix}\right)^*.\nonumber
\end{equation}
Together with $w^*\overline{a}(x)^*(w^*)^{-1} = (a(x),[x,-1])$ and noting how $z\in \Zt$ commutes with other elements we find
\begin{equation}
	a(\xi)^{\varsigma}w^{\varsigma} h_v = (z(\zeta_v),[\zeta_v,-\xi])a(\xi\zeta_v^{-2})^*w^*\left(\begin{matrix} \zeta_v^{-1} & 1 \\ -1 & 0\end{matrix}\right)^*.\nonumber
\end{equation}
Using the transformation properties of $(\psi_v,\mu_v^{\xi})$-Whittaker functions and right-$K_v^*$-invariance we get
\begin{equation}
	W_{\phi, v}^{\mu_v^{\xi}}(a(\xi)^*w^* h_{v}) = [\zeta_v,-\xi]\mu_v^{\xi}(\zeta_v)W_{\phi,v}^{\mu_v^{\xi}}(a(\xi\zeta_v^{-2})^*) = [\zeta_v,-1]\mu_v(\zeta_v)W_{\varphi,v}^{\mu_v^{\xi}}(a(\xi\zeta_v^{-2})^*). \nonumber
\end{equation}
	
Finally, if $v\in S_{\varphi}$, we have
\begin{multline}
	W_{\phi,v}^{\mu_v^{\xi}}(a(\xi)^*w^* h_{v}) =
	\\
	\psi_v(\xi\zeta_v^{-1})\zeta_v(1)^{-\delta_{v\nmid\infty}}\sum_{\mu'\in\Omega(\omega_{\pi_v})}\int_{F_v^{\times}}\psi_v(y\zeta_v)W_{\phi,v}^{\mu_v'}(a(y)^{*}) j_{\pi_v}^{\mu_v^{\xi},\mu_v'}(\xi y)\mu_v(y)^{-1}[-1,y]d^{\times}y.
	\nonumber
\end{multline}
This follows from the original construction
\begin{equation}
	a(\xi)^*w^* h_{v} = n(\xi\zeta_v^{-1})^*a(\xi)^*w^*n(\zeta_v)^*. \nonumber
\end{equation}
The result then follows after re-assembling all the pieces.
\end{proof}

\section{The classical formulation}
\label{sec:classical}

We now recall the adelisation and de-adelsiation procedure and give some interpretations of our Vorono\"i formula in the classical language. We assume some familiarity with the theory of half-integral weight modular forms. Nonetheless we start by recalling some basics, mostly following \cite{shimura-half-integral}.

\subsection{Automorphic forms of half-integral weight}

It is convenient here to give an alternative definition of the metaplectic cover of $\GL_{2}^+(\R)$.
Let $\GLt_{2}^+(\R)$ denote the set of pairs $(g,\phi)$ such that $g=\left(\begin{smallmatrix}
a&b\\c&d
\end{smallmatrix}\right)\in\GL_{2}^+(\R)$ and $\phi\colon \C\rightarrow\C$ is a holomorphic function satisfying $\phi(z)^{2}= \frac{cz+d}{\sqrt{\det(g)}}$. Defining the law of multiplication
\begin{equation}\label{eq:classical-SLt}
(g_{1},\phi_{1}(z))\cdot (g_{2},\phi_{2}(z)) = (g_{1}g_{2}, \phi_{1}(g_{2} z)\phi_{2}(z)),
\end{equation}
we realise $\GLt_{2}^+(\R)$ as the non-trivial two-fold cover of $\GL_{2}^+(\R)$. Explicitly, recalling our choice of the branch cut of the complex square root, one obtains an element $$\epsilon:=\frac{\phi(z)}{(cz+d)^{\frac{1}{2}}}\in Z_{2}$$ such that \eqref{eq:classical-SLt} determines an isomorphism $(g,\phi)\mapsto (g, \epsilon)$ with the group of pairs $(g, \epsilon)\in \GLt_{2}(\R)$ with $\det(g)>0$; see \cite[Lemma 3.3]{gelbart-spectrum}. It will be useful to allow slightly more general co-cycles by only requiring 
\begin{equation}
\phi(z)^{2}= t\frac{cz+d}{\sqrt{\det(g)}}	\label{eq:gen_cocyc}
\end{equation}
for $t\in \{\pm 1\}$.
For the lack of convenient notation we write $\GLt_2^+(\R)$ for the group of tuples $(g,\phi)$, where $\phi$ satisfies the generalised co-cycle identity \eqref{eq:gen_cocyc}. Certainly $\GL_{2}^+(\R)\subset \GLt_2^+(\R)$ and we have the exact sequence $$ 1\to \Mt \to \GLt_2^+(\R) \to \GL_{2}^+(\R)\to 1.$$

For each integer $k\geq 0$, we let the group $\SLt_{2}(\R)$ (and similarly $\GLt_2^+(\R)$) act on the set of meromorphic functions $f\colon \uh\rightarrow\C$ by defining the `half-integral weight slash operator'
\begin{equation}\label{eq:slash-operator-1}
(f\vert_{\frac{k}{2}}\tilde{g})(z) := \left(\frac{\abs{\phi(z)}}{\phi(z)}\right)^{k} \, f(g z)
\end{equation}
for $\tilde{g}=(g,\phi)\in\GLt_{2}^+(\R)$. This action is genuine if and only if $k$ is odd, which we exclusively assume in our construction; otherwise we almost recover the usual integral weight slash operator on $\GL_{2}^+(\R)$.

The prototypical example of a modular form of half-integral weight is the theta function
\begin{equation*}
\vartheta(z)=\sum_{n\in\Z}e(n^{2}z)
\end{equation*}
for $z\in\uh$. Its square is a modular form of weight $1$, level $4$, and character $\chi_{4}$, the primitive Dirichlet character of conductor $4$ such that $\chi_{4}(x)=e(\frac{x-1}{4})$ if $x\in\Z$ is odd. To pick out a lattice in $\SLt_{2}(\R)$, define the co-cycle
\begin{equation}
j_{\vartheta}(\gamma, z ) := \frac{\vartheta(\gamma z)}{ \vartheta(z)} = \overline{\epsilon}_d\left(\frac{c}{d}\right)(cz+d)^{\frac{1}{2}} \nonumber
\end{equation}
for $\epsilon_d=\chi_4(d)^{\frac{1}{2}}$, $\gamma\in \Gamma_{0}(4)$ and $z\in\uh$. Here $\left(\frac{c}{d}\right)$ is the (modified) quadratic residue symbol as defined in \cite[Notation 3]{shimura-half-integral}. For odd $d$ and even $c$ this agrees with $\left(\frac{c}{d}\right)_S$ as defined in \cite[Proposition 2.16]{gelbart-spectrum}. The functional equation for $\theta^{2}$ implies that $$j_{\vartheta}((\begin{smallmatrix}
a&b\\c&d
\end{smallmatrix}), z)^{2}=\chi_4(d)(cz+d).$$
The projection $\SLt_{2}(\R)\rightarrow \SL_{2}(\R)$ splits over $\Gamma_{1}(N)$ whenever $N$ is a positive multiple of $4$. We observe this splitting by virtue of
\begin{equation*}
\Delta_{1}(N):=\left\lbrace \tilde{\gamma} = (\gamma, j_{\vartheta}(\gamma,z))  \, : \, \gamma\in\Gamma_{1}(N)\right\rbrace \subset \SLt_2(\R).
\end{equation*}
Since $j_{\vartheta}(\gamma_{1}\gamma_{2},z)=j_{\vartheta}(\gamma_{1},\gamma_{2}z)j_{\vartheta}(\gamma_{2},z)$, as required by \eqref{eq:classical-SLt}, $\Delta_{1}(N)$ defines a subgroup of $\SLt_{2}(\R)$. We define a modular form of half-integral weight to be a meromorphic function which is stabilised by $\Delta_{1}(N)$ under \eqref{eq:slash-operator-1}. For each integer $k\geq 0$, and $f\colon \uh\rightarrow\C$ we introduce the straightforward extension of the usual slash operation, as in \cite{duke_hyperbolic_distrtibution} (this might be unfamiliar for those used to classical holomorphic modular forms, but is very natural from a spectral point of view):
\begin{equation}\label{eq:slash-operator-2}
(f\vert_{\frac{k}{2}}\tilde{\gamma})(z) := \left(\frac{\abs{j_{\vartheta}(\gamma,z)}}{j_{\vartheta}(\gamma,z)}\right)^k\,f(\gamma z) 
\end{equation}
for $\gamma\in\Gamma_{0}(4)$. Note that, if $\gamma\notin \Gamma_1(N)$, then $\tilde{\gamma}$ is technically not an element in the group $\GLt_2^+(\R)$. However, the slash operation is still well defined precisely due to the more general transformation behaviour of the theta-function. For notational convenience we define
\begin{equation*}
\Delta_{0}(N):=\left\lbrace \tilde{\gamma} = (\gamma, j_{\vartheta}(\gamma,z))  \, : \, \gamma\in\Gamma_{0}(N)\right\rbrace\subset \SLt_2(\R).
\end{equation*}

\begin{defn}\label{def:classical-forms}
Let $k\geq 0$ be an integer and let $N$ be a positive multiple of $4$. Let $\chi$ be a Dirichlet character modulo $N$. We call a smooth function $f\colon \uh\rightarrow \C$ an automorphic form of weight $k+\frac{1}{2}$, level $N$, and character $\chi$ if the following are satisfied:
\begin{itemize}
	\item One has $f\vert_{k+\frac{1}{2}}\tilde{\gamma} = \chi(d) f$ for all $\gamma=\left(\begin{smallmatrix}
		a&b\\c&d
		\end{smallmatrix}\right)\in\Gamma_{0}(N)$.
	\item The function $f$ is a function of moderate growth on $\Gamma_{0}(N)\bs \uh$. 
	\item As a function of $\SLt_{2}(\R)$ it is an eigenfunction of the weight $k+\frac{1}{2}$ Laplace operator
	\begin{equation}\label{eq:casimir}
		\Delta_{k+\frac{1}{2}} :=-y^{2}\left(\frac{\partial^{2}}{\partial x^{2}}+\frac{\partial^{2}}{\partial y^{2}}\right)+i(k+\frac{1}{2})y\frac{\partial}{\partial x}
	\end{equation}
	with eigenvalue $\lambda=s(1-s)$. (i.e $\Delta_{k+\frac{1}{2}}f=\lambda f$.)
\end{itemize}	
Denote the set of such functions $f$ by $\tilde{\mathcal{A}}_{k+\frac{1}{2}}(N,\chi)$. Let $\tilde{\mathcal{A}}^{\circ}_{k+\frac{1}{2}}(N,\chi)$ denote the set of $f\in \tilde{\mathcal{A}}_{k+\frac{1}{2}}(N,\chi)$ such that $\lim_{z\rightarrow \af}f(z)=0$ for each cusp $\af$ of $\Gamma_{0}(N)\bs\uh$. Note that $\tilde{\mathcal{A}}^{\circ}_{k+\frac{1}{2}}(N,\chi)$ is empty unless $\chi(-1)=1$.
\end{defn}

If we wish to consider only the metaplectic cover $\GLt_2^+(\R)$ we may relax the first condition in Definition \ref{def:classical-forms} by assuming only $$f\vert_{k+\frac{1}{2}} \tilde{\gamma} = f \text{ for all }\gamma\in \Gamma_{1}(N),$$ 
whilst keeping the remaining conditions with the necessary modifications. The resulting spaces of such functions are denoted by $\tilde{\mathcal{A}}_{k+\frac{1}{2}}(N)$ and $\tilde{\mathcal{A}}_{k+\frac{1}{2}}^{\circ}(N)$ respectively. One can show that
\begin{equation}
\tilde{\mathcal{A}}_{k+\frac{1}{2}}^{\circ}(N) = \bigoplus_{\chi\text{ mod }N} \tilde{\mathcal{A}}_{k+\frac{1}{2}}^{\circ}(N,\chi). \nonumber
\end{equation}
We are thus naturally led to modular forms that transform with respect to characters under $\Delta_0(N)$.

\subsection{Relation to the literature}

Since our notion of a half-integral weight modular form slightly deviate from the standard ones, let us give some examples of how they relate to the usual classical constructions.

\begin{rem}
We start by recalling the classical definition of weight $\frac{1}{2}$ Maa\ss\  forms as given by S.~Katok and P.~Sarnak in \cite{katok-sarnak}. The theta function use loc.~cit.~is $\tilde{\vartheta}(z) := \Im(z)^{\frac{1}{4}}\vartheta(z).$ The resulting co-cycle is
\begin{equation}
	J(\gamma,z) = \frac{\tilde{\vartheta}(\gamma z)}{\tilde{\vartheta}( z)} = \frac{j_{\vartheta}(\gamma z)}{\abs{cz+d}^{\frac{1}{2}}} = \frac{j_{\vartheta}(\gamma, z)}{\abs{j_{\vartheta}(\gamma, z)}}
\end{equation}
The space considered in \cite{katok-sarnak} now translates into
\begin{align}
	&L^2_{\text{cusp}}(\Gamma_0(N)\backslash\Hf,J) = \tilde{\mathcal{A}}^{\circ}_{\frac{1}{2}}(4,\text{Id}) \nonumber \\
	&= \left\{ f\colon \Hf\to\C\colon \text{square-integrable and cuspidal with }f\vert_{\frac{1}{2}} \gamma = f ,\, \forall \gamma\in\Gamma_0(4)   \right\}.\nonumber
\end{align}
In particular they are interested in eigenfunctions of $-\Delta_{\frac{1}{2}}$.
\end{rem}

\begin{rem}
Even though our notion of an automorphic form includes holomorphic modular forms of half integral weight, it differs slightly from the classical definition. Indeed, a classical modular form $F$ of weight $k+\frac{1}{2}$ is usually defined by requiring 
\begin{itemize}
		\item $F(\gamma z) = \chi(d)j_{\theta}(\gamma,z)^{2k+1} F(z)$ for all $\gamma\in \Gamma_0(N)$,
		\item $F$ is holomorphic in $\Hf$ and at the cusps.
\end{itemize}
Given such a form $F$ we claim that $f(z)=\Im(z)^{\frac{2k+1}{4}}F(z)$ satisfies the definition of a weight $k+\frac{1}{2}$, level $N$ automorphic form of character $\chi$. To see that $f$ exhibits the correct transformation behaviour is straightforward. Using the fact that holomorphic functions are harmonic in that they are annihilated by $\Delta_{\R^2}=\frac{\partial^2}{\partial x^2}+\frac{\partial^2}{\partial y^2}$, we compute that
\begin{equation}
	\Delta_{k+\frac{1}{2}}f = \frac{2k+1}{4}\left(1-\frac{2k+1}{4}\right)f. \label{eq:eig_hol_form}
\end{equation}
Additionally the Cauchy--Riemann equations for $F$ imply that
\begin{equation}
	\Lambda_{k+\frac{1}{2}}f = \left[iy\frac{\partial}{\partial x}-y\frac{\partial}{\partial y}+\frac{2k+1}{4}\right]f = 0.\nonumber
\end{equation}
Thus, $f$ lives in the kernel of the Maa\ss\  weight lowering operator. Moreover, the two conditions \eqref{eq:eig_hol_form} and $\Lambda_{k+\frac{1}{2}}f = 0$ characterise the holomorphy of $F$. This is a nice exercise involving the Cauchy--Riemann equations.
\end{rem}

\subsection{Fourier expansions}\label{sec:fourier-expn}

For a cusp $\af$ of $\Gamma_{0}(N)\bs \uh$ let $\sigma\in\SL_{2}(\Z)$ such that $\sigma \af = \infty$. Suppose that $\gamma\in \Gamma_{0}(N)$ satisfies $\gamma\af = \af$. Then $\sigma \gamma \sigma^{-1}$ fixes the cusp $\infty$ implying it is a matrix of the form $\pm n(x)$ for some $x\in\Z$. Since $\gamma^{-1}\af = \af$, we can make the definition
\begin{equation*}
w(\af):=\min\left\lbrace\, x\geq 1 \, : \, \sigma \gamma \sigma^{-1}= n(x) \text{ for some } \gamma\in\Gamma_{0}(N)\, \right\rbrace.
\end{equation*}
The positive integer $w(\af)$ is called the `width' of the cusp $\af$. It is defined independently of the choice of $\sigma$. If the cusp $\af$ has denominator $q$ (so that $\af=\frac{a}{q}\in\Pb^{1}(\Q)$ with $q\mid N$ and $\ggT(a,N)=1$), it is well known \cite[\S 3.4.1]{nps} that $w(\af)=N/\ggT(q^{2},N)$.

Let $M\mid N$ denote the conductor of $\chi$. For any $x\in\Z$ we have $(f\vert_{k+\frac{1}{2}}\sigma^{-1})(z+xw(\af)) = \chi(1 + axw(\af)q) (f\vert_{k+\frac{1}{2}}\sigma^{-1})(z).$ Note that $\chi(1 + axw(\af)q)= 1$ if and only if $M\mid qw(\af)x$, or equivalently, $\frac{M}{\ggT(qw(\af), M)} \mid x$. Define%\footnote{Note that $\delta(\af)$ equals $w(\af)$ for all cusps $\af$ if and only if the conductor $M$ of $\chi$ divides $N_1$ where $N_1$ is the smallest integer such that $N|N_1^2$.}
\begin{equation*}
\delta(\af):= w(\af) \frac{M}{\ggT(qw(\af), M)} = \frac{\kgV(q^2, N, qM)}{q^2}.
\end{equation*}
Then $\delta(\af)$ is the least integer $d\geq 1$ such that $(f\vert_{k+\frac{1}{2}}\sigma^{-1})(z+d) =  (f\vert_{k+\frac{1}{2}}\sigma^{-1})(z)$. Consequently any $f\in\tilde{\mathcal{A}}_{k+\frac{1}{2}}(N,\chi)$ admits a Fourier expansion at the cusp $\af$ as given by
\begin{equation}\label{eq:cusp-fourier-expansion}
(f\vert_{k+\frac{1}{2}}\tilde{\sigma}^{-1})(z)=\sum_{n\in\Z}a_{f}(n;\af)\,\kappa_{f}(ny/\delta(\af))e(nx/\delta(\af))
\end{equation}
where $\kappa_{f}(\cdot)$ is a Whittaker function depending on the Laplace eigenvalue of $f$ as well as the weight. By separation of variables we show
\begin{equation}
\kappa_{f}(y) = W_{\sgn(y)\frac{2k+1}{4},\frac{s}{2}}(4\pi\abs{y})\nonumber
\end{equation}
for $y\neq 0$. A function $f\in\tilde{\mathcal{A}}_{k+\frac{1}{2}}(N,\chi)$ is cuspidal if and only if $a_{f}(0;\af) = 0$ for all cusps $\af$. We write $a_f(n) = a_f(n;\infty)$. It is important to remember that the coefficients depend on the choice of the scaling matrix $\sigma$. Even though this is a weak dependence it also depends on how we lift $\sigma\in \SL_2(\Z)$ to $\tilde{\sigma}\in\SLt_2(\R)$.

\begin{rem}
	Suppose $f$ has eigenvalue $\lambda = \frac{2k-1}{4}\left(1-\frac{2k-1}{4}\right)$. Then
	\begin{equation}
	\kappa_f(y) = W_{\frac{2k+1}{4},\frac{2k-1}{4}}(4\pi y) = (4\pi y)^{\frac{2k+1}{4}}e^{-2\pi y} \nonumber
	\end{equation}
	for $y>0$. Furthermore, the condition $\Lambda_{k+\frac{1}{2}}f=0$ implies that $a_f(n;\af)=0$ for $n<0$. Therefore one recovers the classical Laurent expansion at the cusps, which is well known for classical holomorphic modular forms of weight $k+\frac{1}{2}$.
\end{rem}

\subsection{Connection to the adelic theory}
\label{sec:connection-to-adeles}

We now determine an embedding of $ \tilde{\mathcal{A}}^{\circ}_{k+\frac{1}{2}}(N,\chi)$ into the subset of adelic automorphic forms. We need the subgroups
\begin{equation}
K_1(N) = \prod_{p<\infty} K_{1,p}(v_p(N));\quad K^1(N) = \prod_{p<\infty} K_{p}^1(v_p(N))
\end{equation}
and
\begin{equation}
K_0(N) = \prod_{p<\infty} K_{0,p}(v_p(N)) \nonumber
\end{equation}
of the adele group $\GL_2(\A)$. (Note that $K^1(N)$ corresponds to $K_0^N$ in the notation of \cite{gelbart-spectrum}.) Our local co-cycles are set-up such that $K^1(N)^*$ defines a subgroup of $\GLt_2(\A)$. The following lemma connects the splitting of $\GL_2(\Q)$ with the (modified) Legendre-symbol.

\begin{lemmy}[Proposition 2.16, \cite{gelbart-spectrum}] \label{lm:comparison_cocycle}
	For all $\gamma=\left(\begin{matrix}a&b\\ c&d\end{matrix}\right)\in \Gamma_1(4)$ one has
	\begin{equation}
	\left(\frac{c}{d}\right) = s(\gamma). \nonumber
	\end{equation}
\end{lemmy}

Next we need a version of strong approximation that applies to $\GLt_2(\A)$.

\begin{lemmy}[Lemma 3.2, \cite{gelbart-spectrum}]
We have
\begin{equation}\label{eq:strong-approx-GL}
	\GLt_{2}(\A)=\GL_{2}(\Q)^{\varsigma}\cdot (\GLt_{2}^+(\R)  \times K^1(N)^*).
\end{equation}
Furthermore, if we write $(g,\zeta) = \gamma^{\varsigma}(g_{\infty},\zeta_{\infty})k^*$ accordingly, then $(g_{\infty},\zeta_{\infty})$ is unique up to left multiplication by $\Delta_1(N)$.
\end{lemmy}

Finally, it will be important to understand how the decomposition of strong approximation is affected by right multiplication by $K_0(N)$.

\begin{lemmy}\label{lm:right_ak_K}
Let $g=\gamma^{\varsigma}(g_{\infty},\zeta_{\infty})k_1^*\in \GLt_2(\A)$ and $k\in K_0(N)\cap K^1(4)$. Then we have 
\begin{equation}
	gk = \tilde{\gamma}^{\varsigma} [(\delta,s(\delta))\cdot(g_{\infty},\zeta_{\infty})]k_2^*\nonumber
\end{equation}
for $\delta \in \Gamma_0(N)$ with $\iota_{fin}(\delta) = k_2k^{-1}k_1^{-1}\in K_0(N)$. 
\end{lemmy}
%\begin{proof}
%For $k\in K^1(4)\cap K_0(N)$ this is essentially trivial. Thus we turn to $k\in K_0(N)\setminus K^1(4)$. By artificially inserting $z(-1)\in \GL_2(\Q)$ we get
%\begin{multline}
%	gk=(-\gamma)^{\varsigma}z(-1)_{\infty}^{*} (g_{\infty},\zeta_{\infty})(-k_1)^*(1,\epsilon) = \tilde{\gamma}^{\varsigma} (-\delta,s(-\delta)\epsilon)z(-1)_{\infty}^*(g_{\infty},\zeta_{\infty})k_2^{*} \\ = \tilde{\gamma}^{\varsigma} (\delta,s(-\delta)\beta_{\infty}(-\delta,z(-1))\epsilon)(g_{\infty},\zeta_{\infty})k_2^{*},\nonumber
%\end{multline}
%for $\epsilon=\beta_v(z(-1),k_v)$ with $v=2$. It remains to explicate the co-cycle. First one checks that for $v=2$ we have
%\begin{equation}
%	\beta_v(z(-1),k_v) = \begin{cases}
%		[l(k_v),-1]_2 &\text{ if } v_2(\gamma)\in 2\Z\cup\{\infty\},\\
%		1&\text {else.}
%	\end{cases} \nonumber
%\end{equation}
%On the other hand we have $\beta_{\infty}(-\delta,z(-1))=-[l(\delta),-1]_{\infty}$. Even more one checks that $$s(-\delta) = s_2(-\delta)s_2(\delta)s(\delta) = \begin{cases}
%	[-1,c]_2s(\delta) &\text{ if } c\neq 0 \text{ and } v_2(c) \text{ odd,}\\
%	1&\text{ else.}
%\end{cases}$$
%\end{proof}

Let $f\in \tilde{\mathcal{A}}_{k+\frac{1}{2}}(N,\chi)$. The adelisation of $\chi$ will be denoted by $\omega_{\chi}$ and we define
\begin{equation}
\omega_{\chi}(k) = \prod_{v \mid M}\omega_{\chi,v}(d_v) \nonumber
\end{equation}
for $k\in K_0(N)$.
Define the function $\varphi_{f}\colon \GLt_2(\A)\to \C$ by
\begin{equation}\label{eq:adelic-embedding}
\varphi_{f}({g}):=\omega_{\chi}(k)(f\vert_{k+\frac{1}{2}}{g}_{\infty})(i)
\end{equation}
where $g=\gamma z(\lambda) g_{\infty} k$ for $\gamma\in \GL_{2}(\Q)^{\zeta}$, $\lambda\in\Rx$, $g_{\infty}\in\SLt_{2}(\R)$, and $k=(\left(\begin{smallmatrix}
a&b\\c&d
\end{smallmatrix}\right),\epsilon)\in K^1(N)$. Let us check several issues. 
\begin{itemize}
	\item $\varphi_f$ is well defined. Indeed, if $g=\gamma_1 z(\lambda_1)g_{\infty,1}k_1 =  \gamma_2 z(\lambda_2)g_{\infty,2}k_2$, then $g_{\infty,1} = \gamma g_{\infty,2}$ for $\gamma\in \Delta_1(N)\subset \SLt_2(\R)$. With this to hand we check that
	\begin{align}
	\varphi_f(\gamma_1z(\lambda_1))g_{\infty,1}k_1)
	&= \omega_{\chi}(k_1)\left[(f\vert_{k+\frac{1}{2}}g_{\infty,1})\right](i) \nonumber\\
	&= \omega_{\chi}(k_1)\left[(f\vert_{k+\frac{1}{2}}\gamma g_{\infty,1})\right](i) \nonumber\\
	&= \omega_{\chi}(k_1)\left[(f\vert_{k+\frac{1}{2}}g_{\infty,2})\right](i)\nonumber\\
	&= \omega_{\chi}(k_1k_2^{-1})\varphi_f(\gamma_2z(\lambda_2))g_{\infty,2}k_2).\nonumber
	\end{align}
	But $\omega_{\chi}(k_1k_2^{-1})=1.$ Thus we have shown that the definition of $\varphi_f$ is independent of choices made whilst using strong approximation.
	
	\item For $k\in K_0(N)\cap K^1(4)$ and $g\in \GLt_{2}(\A)$ we have
	\begin{equation}
	\varphi_f(gk^*) = \omega_{\chi}(k)\varphi_f(g). \nonumber
	\end{equation}
	Indeed by Lemma~\ref{lm:right_ak_K} and the definition we have
	\begin{equation}
		\varphi_f(gk^*) = \omega_{\chi}(k_1^{-1}k_2)\chi(\delta)\varphi_f(g).\nonumber
	\end{equation}
	However $\chi(\delta) = \omega_{\chi}(k_2^{-1}kk_1)$ and the claimed equality follows.
	
	\item Writing $r=qr_{\infty}r_{fin} \in  \Q\R_{>0}\hat{\Z}^{\times}=\A^{\times}$ we find that $(z(r^2),\epsilon)g = (z(q^2)\gamma)\cdot z(\lambda^2r_{\infty}^2)\cdot ((1,\epsilon)g_{\infty})\cdot(z(r_{fin}^2)k)$. According to our adelisation procedure we find
	\begin{equation}
	\varphi_f((z(r^2),\epsilon)g) = \epsilon\omega_{\chi}(z(r_{fin}^2))\varphi_f(g) = \epsilon\omega_{\chi}(r^2)\varphi_f(g). \nonumber 
	\end{equation}
	In other words, $\varphi_f$ is genuine and transforms with respect to the central character $\omega_{\chi}$.
	
	\item Furthermore, we check that by definition we have 
	\begin{equation}
	\varphi_f(gk(\theta)) = e^{i(k+\frac{1}{2})\theta}\varphi_f(g), \nonumber
	\end{equation} 
	for all $g\in \GLt_2(\A)$ and $k(\theta)\in \tilde{\SO}_{2}$. Note that here for $\theta\in [0,4\pi)$ we identify
	\begin{equation}
	k(\theta) = \left[\left(\begin{matrix} \cos(\theta ) & \sin(\theta) \\ -\sin(\theta) & \cos(\theta) \end{matrix}\right), \begin{cases} 1 &\text{ if }\theta \in [0,2\pi),\\ -1 &\text{ if }\theta\in [2\pi,4\pi) \end{cases}\right] \in \widetilde{\SO}_2. \nonumber
	\end{equation}
	In particular, using Iwasawa coordinates we write $$\tilde{g}_{\infty} = \left(\begin{matrix}y^{\frac{1}{2}} & x \\ 0 & y^{-\frac{1}{2}} \end{matrix}\right)^* k(\theta) \in \SLt_2(\R)$$ for $y\in \R_{>0}$, $x\in\R$, and $\theta\in [0,4\pi)$. One thus computes that
	\begin{equation}
	\varphi_f(\tilde{g}_{\infty})= e^{i(k+\frac{1}{2})\theta}f(x+iy). \label{eq:de-adelisation}
	\end{equation}
	
	\item In Iwasawa coordinates the Casimir operator of $\widetilde{\SL}_2(\R)$ is given by
	\begin{equation}
	\Omega = -y^2\left(\frac{\partial^2}{\partial x^2}+\frac{\partial^2}{\partial y^2}\right)+y\frac{\partial^2}{\partial x\partial \theta}. \nonumber
	\end{equation} 
	In particular, according to \eqref{eq:de-adelisation}, we have $\Omega\varphi_f\vert_{\SLt_{2}(\R)} = \Delta_{k+\frac{1}{2}}f=\lambda f$.
	
\end{itemize}
Moreover, it turns out that $\varphi_f$ has moderate growth and, additionally, if $f$ is cuspidal then so is $\varphi_f$. In particular we have an injection
\begin{equation}
\tilde{\Ac}^{\circ}_{k+\frac{1}{2}}(N,\chi) \ni f \mapsto \varphi_f \in \Ac^{\circ}(\GLt_{2}(\A)), \nonumber
\end{equation}
which respects the cuspidal subspaces. Taking all weights $k+\frac{1}{2}$, levels $N$, and characters $\chi$ into account, this exhausts the space $\Ac^{\circ}(\GLt_{2}(\A))$.

\subsection{Connection to automorphic representations}

From now on we assume that $f\in \tilde{\Ac}_{k+\frac{1}{2}}^{\circ}(N,\chi)$ is an eigenfunction of the Hecke operators $T_{p^2}$ for all primes $p\nmid N$. If this is the case, $\varphi_{f}$ generates an (irreducible) genuine cuspidal automorphic representation of $\GLt_{2}(\A)$ in the sense of \cite[Def.~3, p.~58]{gelbart-spectrum}; we denote this representation by $\pi_{f}$. This follows directly from strong multiplicity one. In particular, there exist irreducible, unitary, genuine representations $\pi_{v}$ of $\GLt_{2}(\Q_{v})$, for each place $v$ of $\Q$, such that $\pi_{v}$ is class one for all finite $v\nmid N$; and $\pi_{f}$ decomposes as a restricted tensor product $\pi_{f}=\otimes_{v}\pi_{v}$. The central character of $\pi_f$ is given by $\omega_{\chi}$.

At the archimedean place $v=\infty$ we can give a more detailed description of $\pi_{\infty}$ depending on the type of $f$. To this end, observe that if $f$ is a Maa\ss\ form--which is distinguished by the fact that it is not annihilated by any combination of the Maa\ss\ raising and lowering operators--of weight $k+\frac{1}{2}$ with Laplace eigenvalue $\lambda=s(1-s)$, then 
\begin{equation}
\pi_{\infty}\cong \pi(\sgn^k, s-\frac{1}{2}).\nonumber
\end{equation}
This can be an irreducible principal series if $s-\frac{1}{2}\in i\R$ or a complementary series if $s-\frac{1}{2}\in(0,\frac{1}{2})$. 

Furthermore, let $f=\Im^{\frac{2k+1}{4}}F$ for a classical holomorphic modular form $F$ of weight $k+\frac{1}{2}$; in which case $f$ is annihilated by $\Lambda_{k+\frac{1}{2}}$. Then we have that
\begin{equation}
\pi_{\infty}=\sigma(\sgn^k, k) \nonumber
\end{equation}
is a discrete series representation. 

In all these cases the archimedean component $v_{f,\infty}$ of $f$ is of the form $$ v_{f,\infty}  = (C_f \phi_k,0)$$ when restricted to $\SLt_2(\R)$. Here $\phi_k$ is defined in the induced picture by
\begin{equation}
	\phi_k(k(\theta)) = e^{i(k+\frac{1}{2})\theta}.\nonumber
\end{equation}

\subsection{Fourier vs.~Whittaker expansion}

Consider the adelic Whittaker functions $W_{\varphi}$. If $\varphi$ comes from a classical object, these relate to the standard Fourier coefficients as follows.
Fix the additive character $\psi=\otimes_{v}\psi_{v}$ on $\A/\Q=(\R/\Z)\times\prod_{p<\infty}\Zp$ by $\psi_{\infty}(x)=e(x)$ for $x\in\R$ and $\psi_{p}\vert_{\Zp}=1$ but $\psi_{p}(p^{-1})\neq 1$ for $p<\infty$.

\begin{prop} \label{pr:relation_whitt_fourier}
Let $f\in \tilde{\Ac}_{k+\frac{1}{2}}(N,\chi)$ and let $\af$ be a cusp of $\Gamma_{0}(N)\bs\uh$ with scaling matrix $\sigma\af=\infty$. Then for $\delta\in\Qx$ we have $W_{\varphi_{f}}(a(\delta)^* g_{z}^*\iota_f(\sigma)^*)=0$ unless $\delta=n/\delta(\af)$ for some $n\in\Z$, $n\neq 0$, in which case
\begin{equation*}
	W_{\varphi_{f}}(a(\delta)^* g_{z}^*\iota_f(\sigma)^*)=
	s(\sigma)a_{f}(n;\af)\, \kappa_f(ny/\delta(\af))\,e(nx/\delta(\af))
\end{equation*}
where
$$g_{z}:=n(x)a'(y^{1/2})=\begin{pmatrix}
y^{1/2}&xy^{-1/2}\\&y^{-1/2}
\end{pmatrix}\in \SL_{2}(\R) \text{ with } z=x+iy$$
and we consider the (lifted) scaling matrix $\tilde{\sigma}=(\sigma,(cz+d)^{\frac{1}{2}}) \in \SLt_2(\R)$.
\end{prop}

\begin{proof}
This is proved in \cite[Lemma 3.1]{corbett-saha} in the $\GL_{2}$ case. This proof, taking the correct definition and co-cycle computations into account, may be applied mutatis mutandis to the case at hand for $\GLt_{2}$.
\end{proof}

Note that when $\af=\infty$ we simply recover \cite[Lem.~3, p.~388]{waldspurger-81} which states that
\begin{equation*}
W_{\varphi_{f}}(a(n)) = \kappa_f(n)a_f(n).
\end{equation*}
Suppose that $f$ is a classical cusp form which is an eigenfunction of all Hecke operators. Therefore, as discussed in the previous section, the corresponding adelisation $\varphi_f$ generates a genuine cuspidal automorphic representation $\pi_f$, which factors into local parts. Since we assume that the nebentypus of $f$ satisfies $\chi(-1)=1$, we must have $\omega_{\pi_f,\infty}=1$. This follows from the standard adelisation procedure for Dirichlet characters. In particular we have $\Omega(\omega_{\pi_f,\infty})=\{1,\sgn\}$.

We have $\varphi_{f,\infty}\cong (C_{f}\cdot \phi_k,0)$ and we normalise the Whittaker functionals such that 
\begin{equation}
W_{\varphi_f,\infty}^{1}(a(y)^*) = \delta_{y>0} \kappa_f(y) \text{ and } W_{\varphi_f,\infty}^{\sgn}(a(y)^*) = \delta_{y<0}\kappa_f(y).\label{eq:practical_norm}
\end{equation}
In particular, according to Proposition~\ref{pr:relation_whitt_fourier} we find
\begin{equation}
s(\sigma)a_f(n;\af) = \sum_{\substack{\mu\in \Omega(\omega_{\pi_f}),\\ \mu_{\infty}(-1) = \sgn(n)}}\prod_{v<\infty} W_{\varphi_f,v}^{\mu_v}\left(a\left(\frac{n}{\delta(\af)}\right)^*\iota_f(\sigma)^*\right).\label{eq:fin_part_of_fourier_coeff}
\end{equation}

\subsection{Bessel functions revisited from a classical perspective}
\label{sec:bessel-classical}
With the mapping from classical to adelic automorphic forms given in \S \ref{sec:connection-to-adeles}, here we revise the definition of the adelic Bessel functions to fit the classical context. In particular we need to be careful with the normalisation of the Whittaker functional. In the classical setting we implicitly use Whittaker functionals normalised by \eqref{eq:practical_norm}.
% We refer to this normalisation as the `practical normalisation'. 

We now give an explicit, directly applicable formula for the Bessel transform in the cases needed for the present applications. The missing cases are straightforward to establish using the same recipe.

\begin{lemmy}\label{lm:classical_bessel}
Suppose the Whittaker functionals are normalised as in \eqref{eq:practical_norm} and let $\phi\in \pi_f$ with $W_{\phi,\infty}(a(y)^*)=F(y)$ for $F\in\mathcal{C}^{\infty}(\R)$ with compact support in $\R_{>0}$. Then
\begin{equation}
W_{\phi,\infty}^{\sgn^{\frac{1-\epsilon}2}}(a(\alpha)^*w^*) = [\mathcal{H}_f^{\epsilon,+}F](\alpha) = \int_0^{\infty}\mathcal{J}_f^{\epsilon,+}(\alpha y)F(y)dy.\nonumber
\end{equation}

\begin{itemize}
\item
If $f=\Im^{\frac{2k+1}{4}}F$ for a classical holomorphic modular form $F$ of weight $k+\frac{1}{2}$, then we have $\mathcal{J}_f^{-,+}(x) = 0$ and
\begin{equation}
	\mathcal{J}_f^{+,+}(x) = \delta_{x>0}\sqrt{2}(1-i)\frac{\pi e^{-3\pi i k/2}}{\cos(\pi k)}x^{-\frac{1}{2}}J_{k-\frac{1}{2}}(4\pi\sqrt{x}).\nonumber
\end{equation}

\item
If $f$ is a classical Hecke--Maa\ss\  form of weight $k+\frac{1}{2}$, then we have
\begin{align}
	\mathcal{J}_f^{+,+}(x) &= \delta_{x>0}2x^{-\frac{1}{2}} [K_s((-1)^k\cdot 4\pi i x^{\frac{1}{2}})-iK_s(-(-1)^k\cdot 4\pi i x^{\frac{1}{2}})] \text{ and } \nonumber\\
	\mathcal{J}_f^{-,+}(x) &= \begin{cases}
		\delta_{x<0}2\abs{x}^{-\frac{1}{2}}[(-i)^{-s-1}-i^{-s}]\frac{\Gamma(\frac{1+s}{2}-\frac{k}{2}-\frac{k}{4})}{\Gamma(\frac{1+s}{2}+\frac{k}{2}+\frac{k}{4})}K_s(4\pi\abs{x}^{\frac{1}{2}}), &\text{ if $k$ is even}, \\
		\delta_{x<0}2\abs{x}^{-\frac{1}{2}}[i^{-s-1}+(-i)^{-s}]\frac{\Gamma(\frac{1+s}{2}+\frac{k}{2}+\frac{k}{4})}{\Gamma(\frac{1+s}{2}-\frac{k}{2}-\frac{k}{4})}K_s(4\pi\abs{x}^{\frac{1}{2}}), &\text{ if $k$ is odd}.
	\end{cases} \nonumber
\end{align}
\end{itemize}
\end{lemmy}

\begin{proof}
The existence of the Bessel transform follows directly from Proposition \ref{prop:bessel-real} and the support of $F$. We need only to look up the correct formulae for $j_{\pi_{f,\infty}}^{\pm,+}$ and take care of the correct normalisation. We do so case by case.

If $f$ is a holomorphic modular form of weight $k$, then $\pi_{f,\infty}\cong \sigma(\sgn^k,k)$. In this case only  $j_{\pi_{f,\infty}}^{+,+}$ is non-zero. Moreover, the Bessel functions of the same sign are independent of the normalisation of the Whittaker functionals. Thus the desired formula can be read off directly from \eqref{eq:holo_bessel} with $s=k-\frac{1}{2}$.

If $f$ is  Hecke--Maa\ss\  form of even weight $k$, we again read off the Bessel function with equal signs directly.  For the remaining situation we first make the following observation. If the normalised Whittaker functionals are given by $L_{\text{pr}}^{\pm} = K^{\pm}_f L^{\pm}$ for some constants $K^{\pm}_f$, then
\begin{equation}
	\mathcal{J}_f^{-,+}(x) = \frac{K^-_f}{K^+_f}\cdot j_{\pi_f}^{-,+}(x)\cdot \abs{x}^{-1}.\nonumber
\end{equation}
In practical terms we have $L^{\pm 1}_{\text{pr}}(v_{v,\infty}) = \kappa_f(\pm1).$ On the other hand, we compute
\begin{equation}
	L^{\pm}(\pi(a(\pm 1)^*)v_{f,\infty}) = C_f\int_{\R} \phi_k(w^*n(x)^*)\psi(\mp x)dx = \frac{C_f\pi^{\frac{s+}{2}}}{\Gamma(\frac{1\pm k\pm\frac{1}{2}+s}{2})}\kappa_f(\pm 1).\nonumber
\end{equation}
The last equality is a standard computation. The upshot is that 
\begin{equation}
	K^{\pm}_f= \frac{\Gamma(\frac{1\pm k\pm\frac{1}{2}+s}{2})}{C_f\pi^{\frac{s+1}{2}}}.\nonumber
\end{equation}
The desired Bessel function is now easily determined using the results from \S \ref{sec:bessel-real}.

Finally, if $f$ is a Hecke--Maa\ss\  form of odd weight $k$, then $v_{f,\infty} = C_f(\phi_k,0)$ in the restriction of $\pi(\sgn,s)$ to $\SLt_2(\R)$. However, in \S \ref{sec:bessel-real} we computed the Bessel transforms with respect to the opposite character. This is accounted for by setting
\begin{equation}
	\mathcal{J}_f^{+,+}(x) = j_{\pi_{f,\infty}}^{-,-}(x)\abs{x}^{-1} \text{ and } \mathcal{J}_f^{-,+}(x) = K_f\cdot j_{\pi_{f,\infty}}^{+,-}(x)\abs{x}^{-1}.\nonumber
\end{equation}
Evaluating the constant $K_f$ is similar to the case for even $k$.
\end{proof}

\subsection{Classical \texorpdfstring{Vorono\"i}{Voronoi} formulae}
\label{sec:main-classical}

Here we state and prove our main classical result and a related corollary.

\begin{theorem}\label{thm:classical-Voronoi}
Let $f\in \tilde{\mathcal{A}}^{\circ}(N,\chi)$ (see Definition \ref{def:classical-forms}). For a smooth function $F\colon\R\to\R$ with compact support in $\R_{>0}$, $b\in\N$ and $a\in\Z$ with $\ggT(a,bN)=1$ we have
\begin{multline}
	\sum_{n\in\Z_{\neq 0}}e\bigg(\frac{an}{b}\bigg) a_{f}(n)  F\left(n\right) = \sum_{n\in\Z_{\neq 0}}e\left(-n\frac{\overline{a}}{b\delta(\mathfrak{b})}\right) a_f(n;\mathfrak{b})[\mathcal{H}_f^{\sgn(n),+}F]\left(\frac{n}{\delta(\mathfrak{b})b^2}\right)
	\nonumber
\end{multline}
for the cusp $\mathfrak{b} = \frac{a}{b}$, as per \eqref{eq:cusp-fourier-expansion}, and the Bessel transform $\mathcal{H}_f^{\pm,+}F$ of $F$ are determined by the archimedean type of $f$; see \S \ref{sec:bessel-classical}.
\end{theorem}
\begin{proof}
Fix an isomorphism $\pi=\pi_{\infty}\otimes\pi_{\f}$ under which $\varphi_f=(v_{\varphi_f,\infty},v_{\varphi_f,\f})$. Via the same isomorphism we construct a cusp form $\phi = (v_{\infty},v_{\f})$ determined by
$	v_{\f}= v_{\varphi_f,\f}\nonumber$,
$	W_{v_{\infty}}^1(a(\cdot)) = F(\cdot)$ and
$W_{v_{\infty}}^{\sgn}(a(\cdot))=0.$
The latter determines $v_{\infty}$ in the Kirillov model. We now compute
\begin{align}
	[\pi(n(-\frac{a}{b})^*_{\f})\phi](1) &= \sum_{\alpha\in\Qx} e(\alpha\frac{a}{b})W_{\phi}(a(\alpha)^*)\nonumber \\
	&= \sum_{\alpha\in\Qx}e(\alpha\frac{a}{b})\sum_{\mu\in\Omega(\omega_{\pi})}W_{\phi}^{\mu}(a(\alpha)^*) \nonumber  \\
	&= \sum_{\alpha\in\Q_+}e(\alpha\frac{a}{b})F(\alpha)\sum_{\substack{\mu\in\Omega(\omega_{\pi}),\\ \mu_{\infty}(-1)=1}}W_{\phi,\f}^{\mu}(a(\alpha)^*) \nonumber
	\\
	&= \sum_{\alpha\in\Q_+}e(\alpha\frac{a}{b})F(\alpha)\sum_{\substack{\mu\in\Omega(\omega_{\pi}),\\ \mu_{\infty}(-1)=1}}W_{\varphi_f,\f}^{\mu}(a(\alpha)^*) \nonumber
	\\ 
	&= \sum_{n\in \N} e(n\frac{a}{b})a_f(n;\af)F(n).\nonumber
\end{align}
This is the left-hand side of the equation we want to prove. Due to automorphy of $\phi$ we have the central equality
\begin{multline}
[\pi(n(-\frac{a}{b})^*_{\f})\phi](1) = [\pi(n(-\frac{a}{b})^*_{\f})\phi](w^*) = 
\\
\sum_{\alpha\in\Qx}\sum_{\mu\in\Omega(\omega_{\pi})}W_{\phi,\f}^{\mu}(a(\alpha)^*w^*n(-\frac{a}{b})^*)W_{\phi,\infty}^{\mu_{\infty}}(a(\alpha)^*w^*) .
\nonumber
\end{multline}
This already played a key role in the proof of Proposition~\ref{prop:fundamental-identity}. The desired equality now follows after rewriting the Whittaker expansion of the latter. We start by transforming the archimedean place. Here, according to Lemma~\ref{lm:classical_bessel}, we have
\begin{equation}
	W_{\phi,\infty}^{\sgn^{\frac{1-\epsilon}2}}(a(\alpha)^*w^*)  = [\mathcal{H}_f^{\epsilon,+}F](\alpha).\nonumber
\end{equation}
Here we use that $F$ has support in $\R_{>0}$. Since $j_{\pi_{\infty}}^{\epsilon,+}$ is supported on $\epsilon\R_{>0}$ we have
\begin{align}
	[\pi(n(-\frac{a}{b})^*_{\f})\phi](w^*) &= \sum_{\alpha\in\Qx}\sum_{\epsilon\in\{\pm\}}[\mathcal{H}_f^{\epsilon,+}F](\alpha)\sum_{\substack{\mu\in\Omega(\omega_{\pi}),\\ \mu_{\infty}=\sgn^{\frac{1-\epsilon}{2}}}}W_{\phi,\f}^{\mu}(a(\alpha)^*w^*n(-\frac{a}{b})^*) \nonumber\\
	&= \sum_{\alpha\in\Qx}[\mathcal{H}_f^{\sgn(\alpha),+}F](\alpha)\sum_{\substack{\mu\in\Omega(\omega_{\pi}),\\ \mu_{\infty}(-1)=\sgn(\alpha)}}W_{\varphi_f,\f}^{\mu}(a(\alpha)^*w^*n(-\frac{a}{b})^*). \nonumber
\end{align}
It remains to investigate the finite part of the $\mu$-Whittaker functions. We compute
\begin{equation}
	a(\alpha)^*w^*n(-\frac{a}{b})^* = \left(\left(\begin{matrix} 0 & \alpha  \\ -1 &\frac{a}{b} \end{matrix}\right),\sgn(\alpha)\right) = (\gamma_1,\sgn(\alpha)). \nonumber
\end{equation}
On the other hand we have
\begin{equation}
	z(b^{-1})^*a\left(\alpha b^2\right)^*n\left(\frac{\overline{a}}{b}\right)^* \\ = \left(\begin{matrix}\alpha b & \alpha \overline{a} \\ 0& \frac{1}{b} \end{matrix}\right)^*=\gamma_{2}^*.\nonumber
\end{equation}
Note that $(\gamma_2^*)^{-1} = (\smat{ \frac{1}{\alpha b}} { -\overline{a} }{0}{ b},[b,-\alpha]_{\infty}).$ We can assume without loss of generality that $b>0$ such that $[b,-\alpha]_{\infty}=1$. Artificially we define 
\begin{equation}
	\sigma_{\mathfrak{b}}^* = (\gamma_2,\prod_v s_v(\sigma_{\mathfrak{b}}))^{-1}(\gamma_1,\sgn(\alpha)) =  \left(\begin{matrix}
	\overline{a} & \frac{ 1-a\overline{a}}{b} \\ -b & a
	\end{matrix}\right)^*.\nonumber
\end{equation}
We recognise $\sigma_{\mathfrak{b}}$ as the scaling matrix of the cusp $\frac{a}{b}$. We have set things up so that
\begin{equation}
		W_{\varphi_f,\f}^{\mu} (a(\alpha)^*w^*n(-\frac{a}{b})^*) = s(\sigma_{\mathfrak{b}})\psi_{\f}(\alpha b\overline{a})W_{\varphi_f,\f}^{\mu} (a(\alpha b^2)^*\sigma_{\mathfrak{b}}^*).\nonumber 
\end{equation}
We conclude by using \eqref{eq:fin_part_of_fourier_coeff}.
\end{proof}

\begin{rem}
Note that in the theorem one can always replace $a$ by $a' = a+b\cdot\frac{N}{\ggT(N,a^{\infty}b^{\infty})}$ without changing the left hand side. The upshot is that $a'$ satisfies $\ggT(a',b\delta(\mathfrak{b}))=1$. Of course this modification changes the scaling matrix and the exponential on the resulting right hand side.
\end{rem}

We have essentially used only the adelic language to separate the archimedean part from the Fourier coefficients. This enables us to use the Kirillov model to insert test functions of choice. This theorem has some obvious generalisations. First, one can start from an arbitrary cusp $\af$. In this case the proof remains the roughly same, only the matrix computations become more involved. Second, one can use test functions $F$ with $\text{supp}(f)\cap \R_{<0}\neq \emptyset$. In this situation one will encounter all archimedean Bessel transforms.

To demonstrated the scope of Theorem \ref{thm:classical-Voronoi}, we here derive a Vorono\"i formula which is standard in the $\GL_2$ setting.

\begin{cor}
Suppose $N\mid b$, $\ggT(a,b)=1$ and $b>0$. Then
\begin{multline}
	\sum_{n\in\N}e\bigg(n\frac{a}{b}\bigg) a_{f}(n)   F\left(n\right) = 
	\\
	\overline{\epsilon_{a}}\left(\frac{b}{\overline{a}}\right)\chi_4(a)^k\chi(a)^{-1}\sum_{n\in\Z_{\neq 0}}e\left( -n\frac{\overline{a}}{b}\right) a_f(n)[\mathcal{H}_f^{\sgn(n),+}F]\left(\frac{n}{b^ 2}\right).
	\nonumber  
\end{multline}
\end{cor}
\begin{proof}
We simply apply our main theorem. Recall the cusp data
\begin{equation}
	\mathfrak{b} = \frac{a}{b} \text{ with } \sigma_{\mathfrak{b}} = \left(\begin{matrix}
	\overline{a} & \frac{1-\overline{a}a}{b} \\ -b & a
	\end{matrix}\right).\nonumber
\end{equation}
Since $N\mid b$ we have $\mathfrak{b} = \infty$ and in particular $\delta(\mathfrak{b}) = 1$. On the level of scaling matrices this reduces to $\sigma_{\mathfrak{b}}\in\Gamma_0(N)$. Thus the transformation behaviour of $f$ under $\tilde{\sigma}_{\mathfrak{b}}^{-1}$ implies $$a_f(n,\mathfrak{b}) = \overline{\epsilon_{a}}\left(\frac{b}{\overline{a}}\right)\chi_4(a)^k\chi(a)^{-1}a_f(n).$$ 
\end{proof}

%----------------------------------------------------------------------------------------
%	BIBLIOGRAPHY
%----------------------------------------------------------------------------------------

\bibliographystyle{amsplain}								% stylechoice: amsplain or amsalpha
\bibliography{Metaplectic-refs-last-3}				% calls refs.bib

\end{document}

%----------------------------------------------------------------------------------------
%	THE END
%----------------------------------------------------------------------------------------